\documentclass[a4paper,11pt,leqno]{amsart}
\usepackage[lmargin=3cm,tmargin=3.5cm,bmargin=3.5cm,rmargin=3cm]{geometry} 
\usepackage{amsaddr}
\usepackage{graphicx}
\usepackage{subfig}
\usepackage{amssymb}
\usepackage{amsfonts}
\usepackage{mathrsfs}  
\usepackage{amsmath}
\usepackage{epsfig}
\usepackage{float}
\usepackage{setspace}
\usepackage{tikz}
\usepackage[colorlinks=true, urlcolor=green, citecolor=red, linkcolor=blue, linktocpage,pdfpagelabels, bookmarksnumbered,bookmarksopen]{hyperref}
\newtheorem{theorem}{Theorem}[section]
\newtheorem{lemma}{Lemma}[section]

\newtheorem{remark}{Remark}[section]
\newtheorem{proposition}{Proposition}[section]

\numberwithin{equation}{section}
\newcommand{\R}{\mathbb R}

\begin{document}

\title[Schr\"odinger-Korteweg-de Vries system ]{Well-posedness and long time behavior for the Schr\"odinger-Korteweg-de Vries interactions on the half-Line}

\author[M. Cavalcante]{M\'arcio Cavalcante}
\address{\emph{Instituto de  Matem\'atica\\Universidade Federal de Alagoas - UFAL\\  Macei\'o, Brazil}}
\email{marcio.melo@im.ufal.br}

\author[A. J. Corcho]{Ad\'an J. Corcho}
\address{\emph{Instituto de Matem\'atica\\Universidade Federal do Rio de Janeiro - UFRJ\\ Ilha do Fund\~ao 21941-909.  Rio de Janeiro RJ-Brazil}}
\email{adan@im.ufrj.br}

\thanks{A. J. Corcho was supported by CAPES and CNPq-307761/2016-9, Brazil}



\maketitle

\begin{center}
\begin{minipage}{13cm}
{\small \textbf{Abstract.} The initial-boundary value problem for the Schr\"odinger-Korteweg-de Vries system is considered on the left and right half-line for a wide class of initial-boundary data, including the energy regularity $H^1(\R^{\pm})\times H^1(\R^{\pm})$ for initial data. Assuming homogeneous boundary conditions it is shown for positive coupling interactions that local solutions can be extended globally in time for initial data in the energy space; furthermore, for negative coupling interactions it was proved, 
for a certain class of regular initial data, the following result: if the respective solution does not exhibits finite time blow-up in $H^1(\R^-)\times H^1(\R^-)$, then the norm of the weighted space $L^2\big(\R^-,\, |x|dx\big)\times L^2\big(\R^-,\, |x|dx\big)$ blows-up at infinity time with \textit{super-linear rate}, this is obtained by using a satisfactory algebraic manipulation of a new global virial type identity associated to the system .}
\end{minipage}
\end{center}

\section{\textbf{Introduction}}
The purpose of this paper is to present new results concerning the dynamics of the initial-boundary value problem (IBVP) on the left and right half-line associated 
to the nonlinear interactions modeled by the Schr\"odinger-Korteweg-de Vries (NLS-KdV) system. More precisely, for the right half-line $\mathbb{R}^+=(0, +\infty)$ 
the mathematical model is given by
\begin{equation}\label{SK}
	\begin{cases}
		iu_t+ u_{xx}=\alpha uv+\beta u|u|^2,& (x,t)\in\R^+\times(0,T),\\
		v_t + v_{xxx}+vv_x=\gamma (|u|^2)_x,& (x,t)\in\R^+\times(0,T),\\
		u(x,0)=u_0(x),\ v(x,0)=v_0(x),& x\in \R^+,\\
		u(0,t)=f(t),\ v(0,t)=g(t),& t\in(0,T),
	\end{cases}
\end{equation}
and  for the left half-line $\mathbb{R}^-=(-\infty,0)$ the model is given by
\begin{equation}\label{SKe}
\begin{cases}
iu_t+u_{xx}=\alpha uv+\beta u|u|^2,& (x,t)\in \R^-\times(0,T),\\
v_t+v_{xxx} + vv_x=\gamma(|u|^2)_x,& (x,t)\in \R^-\times(0,T),\\
u(x,0)=u_0(x),\ v(x,0)=v_0(x),& x\in  \R^-,\\
u(0,t)=f(t),\ v(0,t)=g(t),\ v_x(0,t)=h(t),& t\in(0,T), 
\end{cases}
\end{equation}
where in both cases $u=u(x,t)$ is a complex-valued function, $v=v(x,t)$ is a real-valued function and $\alpha,\, \beta, \, \gamma$ are real constants. Also, we  say in both models that the coupling interactions are \textit{positive} or \textit{negative} if $\alpha \gamma >0$ or $\alpha \gamma <0$, respectively. 
\subsection{The evolution of the mass, moment and energy functionals for systems \eqref{SK} and \eqref{SKe}}
The boundary conditions in \eqref{SK} and \eqref{SKe} modify the apparency of the well-known conservation laws for  Schr\"odinger-Korteweg-de Vries interactions posed on the line $\R$. In fact,  for \eqref{SK} and \eqref{SKe} the evolutions of the classical functionals (mass, moment and energy) change substantially as we will see bellow. In order to simplify the notation we denote with super-index $+$ and $-$ the cases right and half-line, respectively. The evolution of the mass is now governed by the following laws:
\begin{equation}\label{i1}
	\mathcal{M}^{\pm}(t):=\int_{\R^{\pm}}|u(x,t)|^2dx=\mathcal{M}^{\pm}(0) \pm 2\,\text{Im}\int_0^tu_x(0,s)\bar{u}(0,s)ds;
\end{equation}
the evolution of the moment obeys the laws:
\begin{equation}\label{i2}
\mathcal{Q}^{\pm}(t):=\int_{\R^{\pm}}\Big(\frac{\alpha}{\gamma}v^2+2\,\text{Im}\big(u\bar{u}_x)\Big)dx
=\mathcal{Q}^{\pm}(0) \mp\int_0^t\big(\mathcal{Q}^u(s) + \mathcal{Q}^v(s)\big)ds,
\end{equation}
where
\begin{align}
&\mathcal{Q}^u(s)=2|u_x(0,s)|^2 + 2\text{Im}\big(u(0,s)\bar{u}_s(0,s)\big)-\beta |u(0,s)|^4, \label{i2-a}\\
&\mathcal{Q}^v(s)=\frac{\alpha}{\gamma}v_x^2(0,s)-\frac{2\alpha}{\gamma} v_{xx}(0,s)v(0,s)-\frac{2\alpha}{3\gamma}v^3(0,s).\label{i2-b}
\end{align}
Finally, the  energy adopts the following expressions:
\begin{equation}\label{i3}
\begin{split}
\mathcal{E}^{\pm}(t)&:=\int_{\R^{\pm}} \Big(\alpha v|u|^2-\frac{\alpha}{6\gamma}v^3+\frac{\beta}{2}|u|^4+\frac{\alpha}{2\gamma}v_x^2+|u_x|^2\Big)dx\\
&=\mathcal{E}^{\pm}(0) \mp \int_0^t\big(\mathcal{E}_1(s) + \mathcal{E}_2(s)\big)ds,
\end{split}
\end{equation}
where
\begin{align}
&\mathcal{E}_1(s)=2\,\text{Re}(u_x(0,s)\bar{u}_s(0,s))+\tfrac{\alpha}{\gamma}v_x(0,s)v_s(0,s),\label{i3-a}\\
&\mathcal{E}_2(s)=\tfrac{\alpha}{2\gamma}\big[\,v_{xx}(0,s)+\tfrac12 v^2(0,s) -\gamma|u(0,s)|^2\,\big]^2.\label{i3-b}
\end{align}

A formal deduction of the laws  \eqref{i1}, \eqref{i2} and \eqref{i3} is exhibited in Section \ref{laws-deduction}.

\begin{remark}
	It is worth pointing out that in the energy identity \eqref{i3} the sign of $\alpha\gamma$ is very important to study the dynamic of models  \eqref{SK} and \eqref{SKe}, since in the work of Corcho and Linares \cite{CL}, in context of all line $\R$,  was obtained global well-posedness results in energy space on the case $\alpha\gamma>0$, while in the case  $\alpha \gamma<0$ the problem of global well-posedness remains open. This motivate us to study the dynamic of solutions of \eqref{SK} and \eqref{SKe} in both cases.
\end{remark}

\subsection{Previous results}
The Cauchy problem for this coupled nonlinear interactions, posed on the line $\R$,  has been studied extensively. For instance, from the physical point of view we refer the works \cite{fi1}, \cite{fi2}, \cite{fi3},  \cite{fi4} and more recently \cite{D-Nuyen-S,Liu-Nguyen}. Concerning well-posedness for the initial value problem (IVP)  with data considered in Sobolev spaces defined on $\R$ and $\mathbb{T}$ we refer the works \cite{Arbieto}, \cite{CL}, \cite{GW}, \cite{YW} and the references contained therein. 

\medskip 
More recently,  in \cite{CM}, the  IBVP \eqref{SK} and \eqref{SKe} was studied for initial-boundary data with the following configuration:
\begin{equation*}\label{hipinicial1}
(u_0,v_0,f,g)\in
\begin{cases}
	\mathscr{H}^{s,\kappa}_+                              & \text{if}\; s, \kappa <1/2,\medskip\\ 
    \mathscr{H}^{s,\kappa}_+\;\;\text{with}\; u_0(0)=f(0) & \text{if}\; s>1/2\; \text{and}\; \kappa <1/2\medskip,\\ 
	\mathscr{H}^{s,\kappa}_+\;\;\text{with}\; v_0(0)=g(0) & \text{if}\; s<1/2\; \text{and}\; \kappa >1/2\medskip,\\ 
	\mathscr{H}^{s,\kappa}_+\;\;\text{with}\; u_0(0)=f(0)\;\text{and}\; v_0(0)=g(0) & \text{if}\; s, \kappa >1/2,
\end{cases}
\end{equation*}
where 
\begin{equation}\label{Hsk-positive}
\mathscr{H}^{s,\kappa}_+:=H^s(\mathbb{R}^+)\times H^{\kappa}(\mathbb{R}^+)\times H^{(2s+1)/4}(\mathbb{R}^+)\times  H^{(\kappa+1)/3}(\mathbb{R}^+).
\end{equation}
This setting was motivated by the following smoothing effects given in \cite{KPV}, namely
\begin{equation}\label{se-1}
	\big\|\psi(t)e^{it\partial_x^2}\phi\big\|_{L^{\infty}_xH^{(2s+1)/4}_t}\leq c_{\psi,s}\|\phi\|_{H^{s}(\mathbb{R})}
\end{equation}
and
\begin{equation}\label{se-2}
	\big\|\psi(t) e^{-t\partial_x^3}\phi\big\|_{L^{\infty}_xH^{(\kappa+1)/3}_t}\leq c_{\psi,\kappa} \|\phi\|_{H^{\kappa}(\mathbb{R})},
\end{equation}
where $\psi(t)$ is a smooth cutoff function and the operators $e^{it\partial_x^2}$ and $e^{-t\partial_x^3}$ denote the free propagator on $\mathbb{R}$ associated to the linear Schr\"odinger and KdV equations, respectively.

\medskip 
Analogously, for the left half-line the configuration considered for initial and boundary data was 
\begin{equation*}\label{hipinicial2}
(u_0,v_0,f,g,h)\in
\begin{cases}
\mathscr{H}^{s,\kappa}_-                              & \text{if}\; s, \kappa <1/2,\medskip\\ 
\mathscr{H}^{s,\kappa}_-\;\;\text{with}\; u_0(0)=f(0) & \text{if}\; s>1/2\; \text{and}\; \kappa <1/2\medskip,\\ 
\mathscr{H}^{s,\kappa}_-\;\;\text{with}\; v_0(0)=g(0) & \text{if}\; s<1/2\; \text{and}\; \kappa >1/2\medskip,\\ 
\mathscr{H}^{s,\kappa}_-\;\;\text{with}\; u_0(0)=f(0)\;\text{and}\; v_0(0)=g(0) & \text{if}\; s, \kappa >1/2,
\end{cases}
\end{equation*}
where
\begin{equation}\label{Hsk-negative}
\mathscr{H}^{s,\kappa}_-:= H^s(\mathbb{R}^-)\times H^{\kappa}(\mathbb{R}^-)\times H^{(2s+1)/4}(\mathbb{R}^+)\times  H^{(\kappa+1)/3}(\mathbb{R}^+)
\times H^{\kappa/3}(\mathbb{R}^+).
\end{equation}

\medskip 
In \cite{CM} the authors obtained a local theory for models \eqref{SK} and \eqref{SKe}, based on the inversion of a Riemann-Liouville fractional integral (see \cite{C}, \cite{CK} and \cite{Holmerkdv}). We summarize these results as follows. 
Local theory in $H^s(\R^+)\times H^{\kappa}(\R^+)$ for the IBVP \eqref{SK} was established in the following situations:

\medskip 
\begin{itemize}
	\item for any data  with   $0\leq s<\frac12$ and  $\max\big\{-\frac34,\ s-1\big\}< \kappa <\min\big\{4s-\frac12,\frac12\big\} $ for all $\beta\in \R$;
	
	\medskip 
	\item for any data  with $\frac12<s<1$ and  $s-1< \kappa <\frac12$ when $\beta=0$; 
	
	\medskip 
	\item  for small KdV-data  with   $\frac14<s<\frac12$ and  $\frac12<\kappa<\min\big\{4s-\frac12,s+\frac12\big\}$  for all $\beta \in \R$;
	
	\medskip 
	\item for small KdV-data with $\frac12<s<1$ and $\frac12 <\kappa <s+\frac12$ when $\beta=0$. 
\end{itemize}
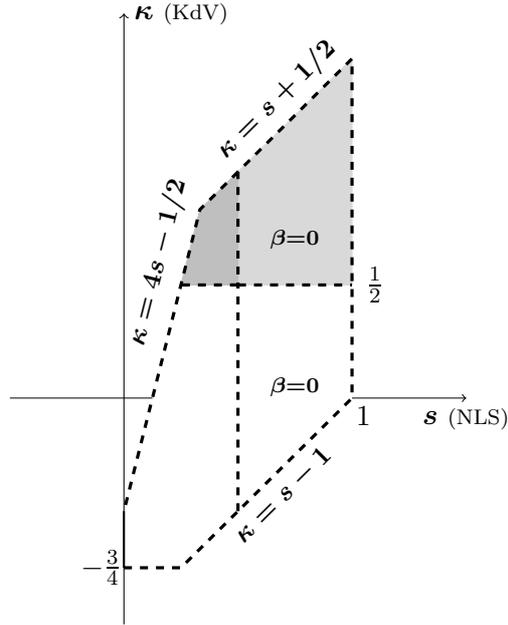
\begin{figure}[htp]\label{figure}
	\centering 
	\begin{tikzpicture}[scale=3]
	\draw[very thin](-0.5,0)--(1/8,0);
	\draw[very thin, ->] (1,0)--(1.5,0) node[below] {$\boldsymbol{s}$ \scriptsize{(NLS)}};
	\draw[->] (0,-1)--(0,1.7) node[right] {$\boldsymbol{\kappa}$ \scriptsize{(KdV)}};
	\filldraw[color=gray!30](1,0.5)--(1,1.5)--(0.5,1)--(0.5,0.5)--(1,0.5);
	\filldraw[color=gray!50](0.5,0.5)--(0.5,1)--(1/3,5/6)--(0.25,0.5)--(0.5,0.5);
	\draw[very thick, dashed](0,-0.75)--(0.25,-0.75)--(1,0)--(1,1.5)--(0.5,1)--(0.5,0.5);
	\draw[very thick, dashed](0.5,1)--(1/3,5/6)--(0,-0.5);
	\draw[very thick, dashed](0.5,-0.5)--(0.5,0.5);
	\draw[very thick, dashed](0.5,0.5)--(1,0.5);
	\draw[very thick, dashed](0.5,0.5)--(0.25,1/2);
	\draw[thick](0,-0.5)--(0,-0.75);
	\node at (0.75,0.05){$\boldsymbol{\scriptstyle \beta=0}$};
	\node at (0.75,0.7){$\boldsymbol{\scriptstyle \beta=0}$};
	\node at (1.05,-0.08){$1$};
	\node at (1.1,0.5){$\frac{1}{2}$};
	\node at (-0.1,-3/4){$-\frac{3}{4}$};
	\node at (0.67,1.32)[rotate=45]{\small{$\boldsymbol{\kappa=s+1/2}$}};
	\node at (0.15,0.6)[rotate=75.96375653]{\small{$\boldsymbol{\kappa=4s-1/2}$}};
	\node at (0.7,-0.43)[rotate=45]{\small{$\boldsymbol{\kappa=s-1}$}};
	\end{tikzpicture}
	\caption{{\small Regions of local well-posedness achieved in \cite{CM} for the right half-line ($\R^+$). For regions painted in gray the respectively local theory was developed under smallness assumption on the data for the KdV-component of the system.}}
	\label{Figura-I}
\end{figure}

Analogous results were obtained in $H^s(\R^-)\times H^{\kappa}(\R^-)$ for the IBVP \eqref{SKe}, which are described in Figure \ref{Figura-II}. 
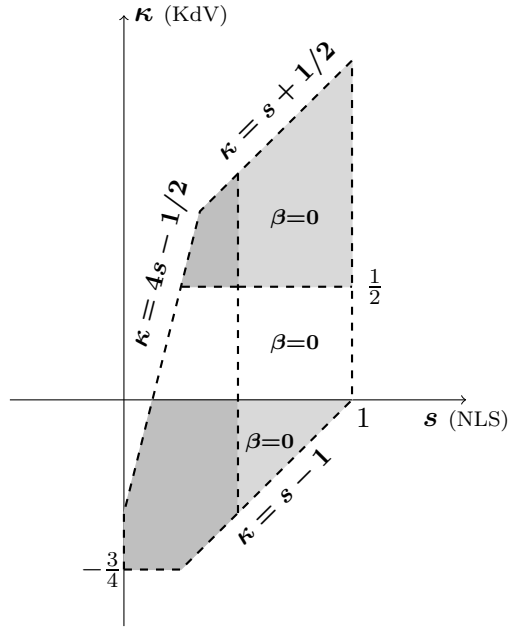
\begin{figure}[h]
	\begin{tikzpicture}[scale=3]
	\draw[->] (-0.5,0)--(1.5,0) node[below] {$\boldsymbol{s}$ \scriptsize{(NLS)}};
	\draw[->] (0,-1)--(0,1.7) node[right] {$\boldsymbol{\kappa}$ \scriptsize{(KdV)}};
	\filldraw[color=gray!30](1,0.5)--(1,1.5)--(0.5,1)--(0.5,0.5)--(1,0.5);
	\filldraw[color=gray!50](0.5,0.5)--(0.5,1)--(1/3,5/6)--(0.25,0.5)--(0.5,0.5);
	\filldraw[color=gray!50](1/8,0)--(0.5,0)--(0.5,-0.5)--(0.25,-3/4)--(0,-3/4)--(0,-0.5)--(1/8,0);
	\filldraw[color=gray!30](0.5,0)--(1,0)--(0.5,-0.5)--(0.5,0);
	\draw[thick, dashed](0,-0.75)--(0.25,-0.75)--(1,0)--(1,1.5)--(0.5,1)--(0.5,0.5);
	\draw[thick, dashed](0.5,1)--(1/3,5/6)--(0,-0.5);
	\draw[thick, dashed](0.5,-0.5)--(0.5,0.5);
	\draw[thick, dashed](0.5,0.5)--(1,0.5);
	\draw[thick, dashed](0,-3/4)--(0,-0.5);
	\draw[thick, dashed](0.5,0.5)--(0.25,1/2);
	\draw[very thin] (1/8,0)--(1,0);
	\node at (0.75,0.25){$\boldsymbol{\scriptstyle\beta=0}$};
	\node at (0.75,0.8){$\boldsymbol{\scriptstyle\beta=0}$};
	\node at (1.05,-0.08){$1$};
	\node at (1.1,0.5){$\frac{1}{2}$};
	\node at (-0.1,-3/4){$-\frac{3}{4}$};
	\node at (0.64,-0.19){$\boldsymbol{\scriptstyle\beta=0}$};
	\node at (0.67,1.32)[rotate=45]{\small{$\boldsymbol{\kappa=s+1/2}$}};
	\node at (0.16,0.6)[rotate=75.96375653]{\small{$\boldsymbol{\kappa=4s-1/2}$}};
	\node at (0.7,-0.43)[rotate=45]{\small{$\boldsymbol{\kappa=s-1}$}};
	\end{tikzpicture}
	\caption{{\small  Regions of local well-posedness achieved in \cite{CM} for the left half-line ($\R^-$). For regions painted in gray the respectively local theory was developed under smallness assumption on the data for the KdV-component of the system.}}
	\label{Figura-II}
\end{figure}

\subsection{Comments about the techniques to solve IBVPS on the half-line} Different techniques have been developed in the last years in order to solve IBVPs associated to some dispersive models on the half line. In \cite{Fokas1} Fokas introduced an approach to solves IBVPs associated to integrable nonlinear evolution equations,  which is known as the unified transform method (UTM) or as Fokas transform method. The UTM method provides a generalization of the Inverse Scattering Transform (IST) method from initial value problems to IBVPs. The  classical method based on the Laplace transform was used successfully in the works \cite{Bona,Bona1,tzirakis,tzirakis3} developed by Bona, Sun, Zhang, Ergodan,  Compaan and Tzirakis.  Recently, a new approach was introduced by Colliander and Kenig \cite{CK} by recasting the IBVP on the half-line by a forced IVP defined in the line $\R$. To see other applications of this technique we refer the results established by Holmer, Cavalcante and Corcho in the works \cite{C,CM,Holmer,Holmerkdv}. On the other hand, in  \cite{Faminskii},  Faminskii used an approach based on the investigation of special solutions of a ``boundary potential'' type for solution of linearized KdV equation  in order to obtain global results for the IBVP associated to the KdV equation on the half-line with more general boundary conditions. More recently, \cite{Fokas2}, Alexandrou, Athanassios, Fokas, Himonas and Mantzavinos introduced a method which combines the UTM method  with a contraction mapping principle.

\subsection{Main results} The results obtained in \cite{CM} do not cover the  energy space regularity  
$H^1(\R^{\pm})\times H^1(\R^{\pm})$, which is the natural environment to study the existence of global solutions as well as other interesting properties concerning the dynamics of the solutions. So, in  this paper we are interested in to establish a local theory for high regularity initial data, including the energy space, and consequently the possibility to show the existence of global solutions in this space under suitable assumptions for the boundary conditions. In this address the same approach used in \cite{CM} can not be applied due the difficulty to estimate the boundary forcing operators in high regularity in the context of Bourgain spaces. So, to establish well-posedness we will follow closely the ideas developed in the works \cite{Bona} and \cite{tzirakis}. In this work is also considered the long time behavior of the solutions. In particular, for the problem posed on the  left half-line, we will prove that for a wide class of data in $H^1(\R^-)\times H^1(\R^-)$ such that the corresponding solutions either blow-up in finite time or are global satisfying 
$$\lim\limits_{t\to +\infty}\frac{1}{t^{1-}}\sup\limits_{t\in [0,\,t]}\left(\big\||x|^{1/2}u(\cdot, t)\big\|_{L^2(\mathbb{R}^-)}+ \big\||x|^{1/2}v(\cdot, t)\big\|_{L^2(\mathbb{R}^-)}\right)=+\infty.$$

\medskip 
Now we are in position to state the main results. We begin by establishing the new local theory for the IBVP \eqref{SK} and \eqref{SKe} including the energy regularity, described in the following theorems:

\begin{theorem}[\textbf{Local theory in  $\R^+$ including energy regularity}]\label{teorema1}
	Consider the IBVP \eqref{SK} with  the compatibility conditions $u_0(0)=f(0)$ and 
	$v_0(0)=g(0)$ and let $s, \kappa$ verifying 
	$$\frac12<s<\frac32\quad \text{and}\quad \frac12< \kappa < \min\Big\{\frac32,\,s+\frac12\Big\}.$$
	Then, for all $(u_0, v_0, f, g)\in \mathscr{H}^{s,\kappa}_+$ there exist a positive time
	$$
	T=T\big(\|u_0\|_{H^s(\R^+)},\|v_0\|_{H^{\kappa}(\R^+)},\|f\|_{H^{(2s+1)/4}(\R^+)},\|g\|_{H^{(\kappa+1)/3}(\R^+)}\big)
	$$ 
	and a solution $(u(\cdot, t),v(\cdot, t)) \in \mathcal{C}\big([0,T];\;H^s(\R^+)\times H^{\kappa}(\R^+)\big)$ of the IBVP \eqref{SK} in the distributional sense 
	satisfying
	\begin{align}
	&\partial_x^ju\in \mathcal{C}(\R_x^+;H^{(2s+1-2j)/4}(0,T))\quad \text{for}\quad j=0,1,\\
	&\partial_x^jv\in \mathcal{C}(\R_x^+;H^{(\kappa+1-j)/4}(0,T))\quad  \text{for}\quad  j=0,1,2.
	\end{align} 
 Moreover, the map $$(u_0,v_0,f,g)\longmapsto (u(\cdot, t),v(\cdot, t))$$
  is locally Lipschitz-continuous from the space  $\mathscr{H}^{s,\kappa}_+$ into the class of the continuous functions  
  $\mathcal{C}\big([0,T];\; H^s(\R^+)\times H^{\kappa}(\R^+)\big)$. 
\end{theorem}

\begin{theorem}[\textbf{Local theory in  $\R^-$ including energy regularity}]\label{teorema1left}  
	Consider the IBVP \eqref{SKe} with  the compatibility conditions 
	$u_0(0)=f(0)$ and $v_0(0)=g(0)$ and let $s, \kappa$ verifying  
	$$\frac12<s<\frac32\quad \text{and}\quad \frac12< \kappa < \min\Big\{\frac32,\,s+\frac12\Big\}.$$
	Then, for all $(u_0, v_0, f, g, h)\in \mathscr{H}^{s,\kappa}_-$ there exist a positive time
	$$
	T=T\big(\|u_0\|_{H^s(\R^-)},\|v_0\|_{H^{\kappa}(\R^-)},\|f\|_{H^{(2s+1)/4}(\R^+)},\|g\|_{H^{(\kappa+1)/3}(\R^+)}, \|h\|_{H^{\kappa/3}(\R^+)} \big)
	$$ 
	and a solution $(u(\cdot, t),v(\cdot, t)) \in \mathcal{C}\big([0,T];\;H^s(\R^-)\times H^{\kappa}(\R^-)\big)$ of the IBVP \eqref{SK} in the distributional sense 
	satisfying
		\begin{align}
		&\partial_x^ju\in \mathcal{C}(\R_x^+;H^{(2s+1-2j)/4}(0,T))\quad  \text{for}\quad  j=0,1,\\
		&\partial_x^jv\in \mathcal{C}(\R_x^+;H^{(\kappa+1-j)/4}(0,T))\quad  \text{for}\quad  j=0,1,2.
		\end{align}
  Moreover, the map 
  $$(u_0,v_0,f,g,h)\longmapsto (u(\cdot, t),v(\cdot, t))$$
  is locally Lipschitz-continuous from the space  $\mathscr{H}^{s,\kappa}_-$ into the class of the continuous functions  
  $\mathcal{C}\big([0,T];\; H^s(\R^-)\times H^{\kappa}(\R^-)\big)$. 
\end{theorem}

\begin{figure}[h]
	\centering 
	\begin{tikzpicture}[scale=3]
	\draw[->](-0.2,0)--(2.25,0) node[below] {$\boldsymbol{s}$ \scriptsize{(NLS)}};
	\draw[->] (0,-0.2)--(0,2) node[right] {$\boldsymbol{\kappa}$ \scriptsize{(KdV)}};
	\filldraw[color=gray!30](0.5,0.5)--(0.5,1)--(1,1.5)--(1.5,1.5)--(1.5,0.5)--(0.5,0.5);
	\draw[line width=0.5pt, dashed](0.5,0.5)--(0.5,1)--(1,1.5)--(1.5,1.5)--(1.5,0.5)--(0.5,0.5);
	\node at (1,-0.10){$1$};
	\node at (0.5,-0.11){$\frac12$};
	\node at (-0.1,0.5){$\frac{1}{2}$};
	\node at (-0.1, 1){$1$};
	\node at (-0.1,1.5){$\frac{3}{2}$};
	\node at (1.5,-0.1){$\frac{3}{2}$};
	\draw[line width=0.25pt, dashed](0,1)--(0.5,1);
	\draw[line width=0.25pt, dashed](1,0)--(1,0.5);
	\draw[line width=0.25pt, dashed](0,0.5)--(0.5,0.5)--(0.5,0);
	\draw[line width=0.25pt, dashed](1.5,0)--(1.5,0.5);
	\draw[line width=0.25pt, dashed](0,1.5)--(1.5,1.5);
	\node at (0.4,1.32){\small{$\boldsymbol{\kappa=s+1/2}$}};
	\node at (1,1){$\bullet$};
	\node at (1,0.9){\small{$H^1_{\pm}\times H^1_{\pm}$}};
    \end{tikzpicture}
	\caption{{\small Regions of local well-posedness achieved in Theorems \ref{teorema1} and \ref{teorema1left}, where $H^1_{\pm}:=H^1(\R^{\pm})$}.}
	\label{Figura-III}
\end{figure}
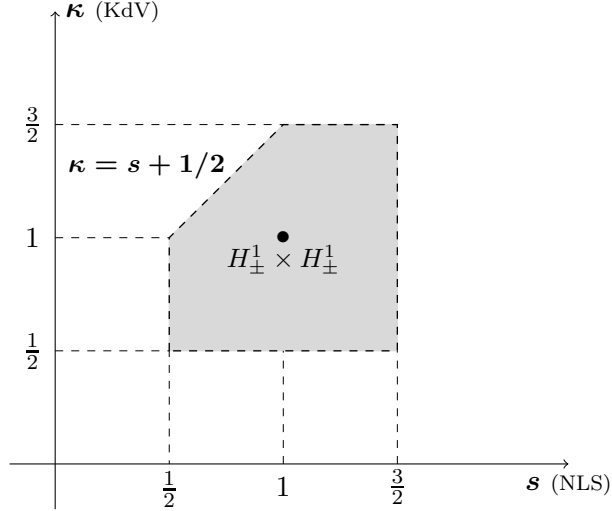
\medskip

To obtain  the results established in theorems \ref{teorema1} and \ref{teorema1left} we will follow the ideas contained in \cite{Bona1}, \cite{CL}, \cite{tzirakis} and \cite{Faminskii}, based on the Laplace transform and potential theory combined with the Fourier restriction method.  

\medskip 
Concerning global solutions we prove that positive coupled interactions on the  right half-line  ($\alpha \gamma >0$), under homogeneous boundary conditions, ensure that local solutions in the energy space can be extended for all time interval $[0, T]$. This global result follows from the laws \eqref{i1}, \eqref{i2} and \eqref{i3}, which under 
conditions $u(0,t)=g(0,t)=0$ take the following forms:
\begin{equation}\label{i10}
\mathcal{M}^{\pm}(t)=\mathcal{M}^{\pm}(0),
\end{equation}
\begin{equation}\label{i20}
\mathcal{Q}^{\pm}(t)=\mathcal{Q}^{\pm}(0)\mp\int_0^t\Big(2|u_x(0,s)|^2 +\frac{\alpha}{\gamma}v_x^2(0,s)\Big)ds
\end{equation}
and
\begin{equation}\label{i30}
\mathcal{E}^{\pm}(t):=\mathcal{E}^{\pm}(0) \mp\frac{\alpha}{2\gamma}\int_0^{t}v^2_{xx}(0,s)ds.
\end{equation}

\medskip
In what follows we use the notation  $\mathcal{M}^{\pm}(0)=\mathcal{M}^{\pm}_0$, $\mathcal{Q}^{\pm}(0):=\mathcal{Q}^{\pm}_0$ and $\mathcal{E}^{\pm}(0):=\mathcal{E}^{\pm}_0$.

\medskip 
The identities \eqref{i10}, \eqref{i20} and \eqref{i30} can be justified for solutions $u(\cdot, t)$ and $v(\cdot, t)$ in $H^1(\R^{\pm})$ by the standard regularization procedure; see for example the proof of \eqref{i10} in Section 
\ref{mass-homo-regularization}.

\medskip 
The global results in $H^1(\R^+)\times H^1(\R^+)$ reads as follows:

\begin{theorem}[\textbf{Long-time behavior in $\R^+$}]\label{Th-global-theory}
Let $\alpha, \gamma \in \R$ such that $\alpha\gamma>0$ and $(u_0,v_0)$ in the energy space $H^1(\mathbb{R}^+)\times H^1(\mathbb{R}^+) $. Then, the corresponding local solution of the IBVP \eqref{SK}  with  homogeneous boundary conditions (i.e. $f=g=0$) can be extended to all time interval $[0, T]$. In addition, for more regular data verifying 
\begin{equation}\label{Th-global-theory-a}
u_0\in L^2\big(\R^+,xdx\big) \quad \text{and} \quad \mathcal{Q}^+_0 <0,
\end{equation}
the respective solution for the Schr\"odinger component satisfies
\begin{equation}\label{Th-global-theory-b}
\big \|x^{1/2} u(\cdot, t)\big \|^2_{L^2(\R^+)}\geq |\mathcal{Q}^+_0| t+ \big \|x^{1/2} u_0\big \|^2_{L^2(\R^+)};
\end{equation}
so  $\lim\limits_{t \to +\infty}\big \|x^{1/2} u(\cdot, t)\big \|_{L^2(\R^+)}=+\infty$. 
\end{theorem}

\medskip 
Finally, we establish the more technical result obtained in this work, which concerns with the long-time behavior for the solutions in the weighted space 
$L^2\big(\R^-,\, |x|dx\big)\times L^2\big(\R^-,\, |x|dx\big)$ for negative interactions ($\alpha\gamma < 0$) and also with homogeneous boundary conditions. 

\begin{theorem}[\textbf{Long-time behavior in $\R^-$}]\label{Th-grow-time-left} 
	Consider the IBVP \eqref{SKe} with $\alpha \gamma<0$. 	Let $(u_0, v_0)$ a regular data in $H^1(\mathbb{R}^-)\times H^1(\mathbb{R}^-)$ such that 
	$$\big \||x|^{1/2}u_0 \big\|_{L^2(\mathbb{R}^-)} + \big\||x|^{1/2}v_0\big\|_{L^2(\mathbb{R}^-)} < \infty$$
	and suppose that the corresponding solution is defined for all  time.
	Then, 
	\begin{enumerate}
		\item[(a)] if\; $\mathcal{Q}^-_0>0$\; we have\;
		$\big\||x|^{1/2}u\big\|^2_{L^2(\mathbb{R}^-)}\ge \mathcal{Q}_0^-t + \big\||x|^{1/2}u_0\big\|^2_{L^2(\mathbb{R}^-)}.$
		
		\medskip 
		\item[(b)] if\; $ \mathcal{Q}^-_0 > 8\mathcal{E}^{-}_0$ and $\beta \ge 2\big|\alpha \gamma|$ we have
		$$\lim\limits_{t\to +\infty}\frac{1}{t^{1-}}\sup\limits_{t\in [0,\,t]}\left(\big\||x|^{1/2}u(\cdot, t)\big\|_{L^2(\mathbb{R}^-)}+ \big\||x|^{1/2}v(\cdot, t)\big\|_{L^2(\mathbb{R}^-)}\right)=+\infty.$$
	\end{enumerate}
\end{theorem}

\begin{remark}
Notice  that the conclusion in (b) above implies that
$$\lim\limits_{t\to +\infty}\big(\,\big\||x|^{1/2}u(\cdot, t)\big\|+ \big\||x|^{1/2}v(\cdot, t)\big\|\,\big)=+\infty.$$
\end{remark}

\medskip 
To prove Theorem \ref{Th-grow-time-left} we will apply the virial technique developed by Glassey  in  \cite{glassey} in the context of nonlinear Schr\"odinger equations (NLS). The functional variance
\begin{equation}
t\longmapsto \int_{\R} |x|^2|u(x,t)|^2dx,
\end{equation}
used in the NLS case, will be replaced here by
\begin{equation}\label{sign}
t\longmapsto \eta(t):=\int_{\R^-} |x|^2|u(x,t)|^2dx +\int_{0}^t\int_{\R^-}x|u|^2(x,t)dx -\frac{2\alpha}{\gamma}\int_{0}^t\int_{\R^-}xv^2(x,t)dx.
\end{equation}

A crucial point to note here is that our approach doesn't work in context of all line $\R$, since in this case we wouldn't have a defined sign for the corresponding 
integral terms $\displaystyle \int_{\R}xv^2(x,t)dx$ and  $\displaystyle \int_{\R}x|u|^2(x,t)dx$.

\begin{remark} We now point out two important remarks:
	
\medskip 
\begin{enumerate}
\item [(a)] We believe that the same approach used to prove Theorem \ref{Th-grow-time-left} can provide similar results for other nonlinear dispersive systems on the half-line, for example, the quadratic Schr\"odinger system \cite{Barbosa} or the coupled multicomponent NLS-gKorteweg-de Vries system \cite{BCP}.\medskip

\item [(b)] The problem about the possible blow-up of $H^1$-norm for solutions to the IBVP problems \eqref{SK} and \eqref{SKe}, under certain initial-boundary conditions,
remains open. 
\end{enumerate}
\end{remark}

We finally notice that there exists in the literature some  results concerns the study of the dynamic for some nonlinear dispersive equations on the half-line. For instance, about the nonlinear Schr\"odinger equation we cite the works of Kalantarov and \"Ozsari \cite{Ozsari} and Erdogan and Tzirakis \cite{Tzirakis}. Finally, in \cite{Dias1} and \cite{Dias} was studied the dynamics of solutions for the Benney system.

\subsection{Structure of the paper.} This work is organized as follow: in the next section we discuss notations, introduce important function spaces and review  
some properties of the linear group associated to the Schr\"odinger and Korteweg - de Vries equation. Local theory  is developed in \ref{3}. Sections \ref{4} and \ref{5} are devoted to the proof  of theorems \ref{Th-global-theory} and \ref{Th-grow-time-left}, respectively.

\section{\textbf{Preliminaries}}
First we introduce some notations and functions spaces. The $L^2$ - based Sobolev space on the line $\R$ is denoted by $H^s(\R)$ with the norm
$\|\phi\|_{H^s(\mathbb{R})}=\|\langle\xi\rangle^s\widehat{\phi}(\xi)\|_{L^2(\mathbb{R})}$, where $\langle\xi\rangle=1+|\xi|$ and $\widehat{\phi}$ denotes the Fourier transform of
of $\phi$: 
$$\displaystyle \widehat{\phi}(\xi)=\int_{\R} e^{-i\xi x}\phi(x)dx.$$
The norm in the homogeneous space $\dot{H}^s(\R)$ is given by $\|\phi\|_{\dot{H}^s(\R)}=\|\mathscr{D}^s\phi\|_{L^2(\R)}$, where
$\widehat{\mathscr{D}^s\phi}(\xi)=\xi^s\widehat{\phi}(\xi).$ Also, 
$$\widehat{u}(\xi,\tau)=\int_{\R^2} e^{-i(\xi x+\tau t)}u(x,t)dxdt$$
denotes the space-time Fourier transform of $u$; eventually we use $\mathscr{F}_{x}u(\xi,t)$  and $\mathscr{F}_{t}u(x,\tau)$  to denote the space and times Fourier transform of $u$, respectively. 

\medskip 
The notations $f \lesssim  g$ and $f \gtrsim  g$ means that there is a positive constant $C$ such that the functions $f$ and $g$ verify the relations  $|f| \leq C |g|$ and 
$|f| \ge C |g|$, respectively. The characteristic function of an arbitrary set $A$ is denoted by $\chi_{A}$ and throughout the paper, we fix a cutoff function
$\psi \in \mathcal{C}_0^{\infty}(\mathbb{R})$ such that
$$
\psi(t)=
\begin{cases}
1 & \text{if}\; |t|\le 1,\\
0 & \text{if}\; |t|\ge 2
\end{cases}
$$ 
and $\psi_{\delta}(t):=\psi(\frac{t}{\delta})$ for any $\delta >0$.

\medskip 
Let a nonnegative number $s\geq 0$. We say that $\phi \in H^s(\mathbb{R}^+)$ if exists $\tilde{\phi}\in H^s(\mathbb{R})$ such that 
$\phi=\tilde{\phi}\big|_{\R+}$  a.e.  In this case we set $\|\phi\|_{H^s(\mathbb{R}^+)}:=\inf\limits_{\tilde{\phi}}\|\tilde{\phi}\|_{H^{s}(\mathbb{R})}.$
Moreover, we say that  $\phi \in H_0^s(\mathbb{R}^+)$ if $\phi \in H^s(\R^+)$ and $\text{supp} (\phi) \subset[0,+\infty)$. For $s<0$, $H^s(\mathbb{R}^+)$ and $H_0^s(\mathbb{R}^+)$  are defined as the dual space of $H_0^{-s}(\mathbb{R}^+)$ and  $H^{-s}(\mathbb{R}^+)$, respectively.

\medskip 
Moreover we consider the classes
$$\mathcal{C}_0^{\infty}(\mathbb{R}^+)=\big\{\phi\in \mathcal{C}^{\infty}(\mathbb{R});\; \text{supp}(\phi) \subset [0,+\infty)\big\}$$
and $\mathcal{C}_{0,c}^{\infty}(\mathbb{R}^+)$ as those members of $\mathcal{C}_0^{\infty}(\mathbb{R}^+)$ with compact support and also we recall 
that the space $\mathcal{C}_{0,c}^{\infty}(\mathbb{R}^+)$ is dense in $H_0^s(\mathbb{R}^+)$ for all $s\in \mathbb{R}$.

\medskip 
In the case of left half-line, the dentition for $H^s(\R^-)$ and  $H^s_0(\R^-)$ is given in a similar way to cases $H^s(\R^+)$ and  $H^s_0(\R^+)$.

\medskip 
The following results summarize useful properties of the Sobolev spaces on the half-line and we refer  \cite{CK} for the proofs.
\begin{lemma}\label{sobolev-estimates}
For all $\phi \in H^s(\mathbb{R})$ with  $-\frac{1}{2}<s<\frac{1}{2}$  we have
	\begin{equation}\label{sobolev-estimates-a}
	\|\chi_{\R^+}\phi\|_{H^s(\mathbb{R})}\le c_s\|\phi\|_{H^s(\mathbb{R})}.
	\end{equation}
Furthermore, if $0\leq s<\frac{1}{2}$ it holds that
\begin{align}
	&\|\psi \phi\|_{H^s(\mathbb{R})}\leq c_{\psi, s} \|\phi\|_{\dot{H}^{s}(\mathbb{R})},\label{sobolev-estimates-b}\\
	&\|\psi \phi\|_{\dot{H}^{-s}(\mathbb{R})}\leq c_{\psi, s} \|\phi\|_{H^{-s}(\mathbb{R})}\label{sobolev-estimates-c}.
\end{align}
\end{lemma}
\begin{lemma}\label{sobolev-estimates-0}
For all $\phi \in H_0^s(\mathbb{R}^+)$ with $s\in \R$ we have 
\begin{equation}\label{sobolev-estimates-0-a}
\|\psi \phi\|_{H_0^s(\mathbb{R}^+)}\leq c \|\phi\|_{H_0^s(\mathbb{R}^+)}.
\end{equation}	
Furthermore, if $\frac{1}{2}<s<\frac{3}{2}$ we have $H_0^s(\R^+)=\big\{\phi\in H^s(\R^+);\, \phi(0)=0\big\}$ and it holds that 
\begin{equation}\label{sobolev-estimates-0-b}
\|\chi_{\R^+}\phi\|_{H_0^s(\R^+)}\leq c_s \|\phi\|_{H^s(\R^+)}.
\end{equation}
\end{lemma}

We denote by $X^{s,b}$  the so called Bourgain spaces associated to linear Schr\"odinger equation, more precisely, $X^{s,b}$ is the completion of the Schwartz space 
$\mathscr{S}'(\R^2)$ with respect to the norm
\begin{equation*}\label{Bourgain-norm}
\|w\|_{X^{s,b}}=\|\langle\xi\rangle^s\langle\tau+\xi^2\rangle^b\widehat{w}(\xi,\tau) \|_{L_{\tau}^2L^2_{\xi}}.
\end{equation*}
To obtain the local theory  we also need to define the following auxiliary modified Bourgain spaces introduced by Faminskii \cite{Faminskii} in the context of Korteweg-de Vries equation. For $\kappa\geq 0$,\,  $0< b <\frac12$ and  $\frac12 < \alpha < \frac23$ define the space $Y^{\kappa,b,\alpha}$ and $W^{\kappa,b,\alpha}$  as the completion of $\mathscr{S}'(\mathbb{R}^2)$ with respect to the norms
\begin{equation}
\|w\|_{Y^{\kappa,b,\alpha}}=\big\|\big(1+|\xi|+|\tau|^{\frac13}\big)^{\kappa}\mu(\xi,\tau)\hat{w}(\xi,\tau)\big\|_{L_{\tau}^2L^2_{\xi}},
\end{equation}
where
\begin{equation*}
\mu(\xi,\tau)=(1+|\tau-\xi^3|)^b+\chi_{[-1,1]}(\xi)(1+|\tau|)^{\alpha}
\end{equation*}
and
\begin{equation}
\|w\|_{W^{\kappa,b,\alpha}}=\big\|\big(1+|\xi|+|\tau|^{\frac13}\big)^{\kappa}\nu(\xi,\tau)\hat{w}(\xi,\tau)\big\|_{L_{\tau}^2L^2_{\xi}},
\end{equation}
where
\begin{equation*}
\nu(\xi,\tau)=\frac{1}{(1+|\tau-\xi^3|)^b}+\frac{\chi_{[-1,1]}(\xi)}{(1+|\tau|)^{1-\alpha}}
\end{equation*}

\medskip 
Next we summarize some useful estimates valid in these spaces.

\begin{lemma}\label{Xsb-estimates-GTV}
Let $s\in \R$ and $b, b'$ real numbers verifying the relations: $-\frac{1}{2}< b'< b\leq 0$  or $0\leq b'<b<\frac{1}{2}.$
Then, there exists a positive constant only depending on $\psi, s, b$ and  $b'$ such that 
\begin{equation*}
\|\psi_{T}w\|_{ X^{s,b'}}\leq c T^{b-b'}\|w\|_{ X^{s,b}}
\end{equation*}
for any $w\in  X^{s,b}$. 
\end{lemma}

The estimates in Lemma \ref{Xsb-estimates-GTV} was proved  by Ginibre, Tsutsumi and Velo in \cite{GTV}, where they esta\-blished well-posednes for the Zakharov system using the Fourier restriction method.

\begin{lemma}\label{faminski1}
Let $\kappa \ge 0$, $0< b<\frac12$ and $\frac12< \alpha <\frac23$. Then, for any  positive time  $T>0$ we have
\begin{equation*}
\|\psi_T(t)w\|_{Y^{\kappa,b,\alpha}}\lesssim T^{\frac12-\alpha-\frac{\kappa}{3}}\|w\|_{Y^{\kappa,b,\alpha}}
\end{equation*}
for any $w\in Y^{\kappa, b, \alpha}$.
\end{lemma}
\begin{proof}
This estimate was proved by Faminskii in Lemma 2.1 of \cite{Faminskii}.  
\end{proof}

\subsection{Estimates for the free propagators} The linear group associated to the linear Schr\"odinger equation is given by te operators
$$e^{it\partial_x^2}: \mathscr{S}'(\R)\rightarrow  \mathscr{S}'(\R),$$
where
$\mathscr{F}_{x}\big[e^{it\partial_x^2}\phi\big](\xi)=e^{-it\xi^2}\widehat{\phi}(\xi)$,
so that
\begin{equation}\label{grupo-s}
\begin{cases}
(i\partial_t+\partial_x^2)e^{it\partial_x^2}\phi(x) =0,& (x,t)\in\mathbb{R}\times\mathbb{R},\\
e^{it\partial_x^2}\phi(x)\big|_{t=0}=\phi(x),& x\in\mathbb{R}.
\end{cases}
\end{equation}

The next result collect some important estimates for $e^{it\partial_x^2}$. 
 
\begin{lemma}\label{lemma-grupo-s}
	Let $s\in\mathbb{R}$ and  $0< b<1$. Then for all $\phi\in H^s(\mathbb{R})$ we have the following estimates:
	\begin{align}
		&\|e^{it\partial_x^2}\phi(x)\|_{\mathcal{C}\big(\mathbb{R}_t;\,H^s(\mathbb{R}_x)\big)}\leq c_s\|\phi\|_{H^s(\mathbb{R})},\label{lemma-grupo-s-a}\\
		&\|\psi(t) e^{it\partial_x^2}\phi(x)\|_{\mathcal{C}\big(\mathbb{R}_x;\,H^{(2s+1)/4}(\mathbb{R}_t)\big)}\leq c_{\psi, s} \|\phi\|_{H^s(\mathbb{R})},\label{lemma-grupo-s-b}\\
		&\|\psi(t)e^{it\partial_x^2}\phi(x)\|_{X^{s,b}}\leq c_{\psi, s}\|\psi\|_{H^1(\mathbb{R})} \|\phi\|_{H^s(\mathbb{R})}.\label{lemma-grupo-s-c}
	\end{align}
\end{lemma}
The proof of Lemma \ref{lemma-grupo-s}  can be seen in \cite{Holmer}, work developed by Holmer when he studied the Schr\"odinger equation on the half-line

\medskip
Similarly, the linear unitary  group associated to the linear KdV equation is defined by
$$e^{-t\partial_x^3}:\mathscr{S}'(\R)\rightarrow \mathscr{S}'(\R),$$
where
$\mathscr{F}_{x}\big[e^{-t\partial_x^3}\phi\big](\xi)=e^{it\xi^3}\widehat{\phi}(\xi)$,
so that
\begin{equation}\label{grupo-kdv}
\begin{cases}
(\partial_t+\partial_x^3)e^{-t\partial_x^3}\phi(x,t)=0& \text{for}\quad (x,t)\in \mathbb{R}\times\mathbb{R},\\
e^{-t\partial_x^3}\phi(x)\big|_{t=0}=\phi(x)&\text{for}\quad  x\in\mathbb{R}.
\end{cases} 
\end{equation}

The next estimates were used by Famiskii in \cite{Faminskii}, where he studied the Korteweg-de Vries equation in half-strip with data in fractional-order Sobolev spaces.
\begin{lemma}\label{lemma-grupo-kdv}
Let $\kappa\ge 0$, $0< b<1$ and $\frac12< \alpha <\frac23$. Then for all  $\phi\in H^{\kappa}(\R)$ we have the following estimates:
\begin{align*}
	&\|e^{-t\partial_x^3}\phi(x)\|_{\mathcal{C}\big(\mathbb{R}_t;\,H^{\kappa}(\mathbb{R}_x)\big)}\leq c_{\kappa}\|\phi\|_{H^{\kappa}(\mathbb{R})},\\
	&\|\psi(t) e^{-t\partial_x^3}\phi(x)\|_{\mathcal{C}\big(\mathbb{R}_x;\,H^{(\kappa +1)/3}(\mathbb{R}_t)\big)}
	\leq c_{\psi, \kappa} \|\phi\|_{H^{\kappa}(\mathbb{R})},\\
	&\|\psi(t)e^{-t\partial_x^3}\phi(x)\|_{Y^{\kappa,b,\alpha}}\leq c_{\psi, \kappa} \|\phi\|_{H^{\kappa}(\mathbb{R})}.
\end{align*}
\end{lemma}

\section{\textbf{Well-posedness including energy regularity}}\label{3}
The problem \eqref{SK} was studied for low regularity in \cite{CM} by using the ideas contained in \cite{C}, \cite{CK} and \cite{Holmerkdv}, based in a inversion of a 
Riemann-Liouville fractional integral. Here we are interesting in the dynamics of the solutions in the energy space $H^1(\R^{\pm})\times H^1(\R^{\pm})$, so 
we need to justify a local theory for the IBVP \eqref{SK} in high regularity assumptions. The estimates on Bourgain spaces of the boundary operator given in \cite{CM} 
do not reach high regularity, so we need to use another method. In this direction, the crucial first step is to obtain an appropriate operators that solve the following 
linear versions of the IBVPs:
\begin{equation}\label{linearnls}
\begin{cases}
iu_t+u_{xx}=0,& x,t\in \R^{+},\\
u(x,0)=0,& x \in \R^+,\\
u(0,t)=f(t)\in H^{(2s+1)/4}(\R^+)&
\end{cases}
\end{equation} 
and
\begin{equation}\label{linearkdv}
\begin{cases}
v_t+v_{xxx}=0,& x,t\in \R^{+},\\
v(x,0)=0& x \in \R^+,\\
v(0,t)=g(t)\in H^{(\kappa+1)/3}(\R^+),&
\end{cases}
\end{equation} 
having a nice behavior in Bourgain spaces wit high regularity.

\subsection{Linear operators}
To solve the linear IBVP \eqref{linearnls} we consider the boundary operator given in \cite{Bona1} and \cite{tzirakis}, based on the Laplace transform. More precisely, we use here the operator
\begin{equation}\label{L-operator}
\mathcal{L}f(x,t)=\mathcal{L}_{-1}f(x,t)+\mathcal{L}_1f(x,t),
\end{equation}
where
\begin{align*}
&\mathcal{L}_{\lambda}q(t)=\frac{1}{\pi}\int_0^{+\infty}e^{\lambda(iz^2t-iz x)} z\, \widehat{q}(\lambda z^2)dz,\quad \lambda=\pm 1,\\
\intertext{and}
&\widehat{q}(\tau)=\mathcal{F}(\chi_{(0,+\infty)}f)(\tau)=\int_{0}^{+\infty}e^{-i\tau t}f(t)dt,
\end{align*}
with $f\in \mathcal{C}_0^{\infty}(\R^+)$. We refer the reader to \cite{Bona1} to see the derivation of operator $\mathcal{L}$.

\medskip 
On the other hand, to solve the second linear IBVP \eqref{linearkdv} we use the approach of Faminskii (see \cite{Faminskii}), based in a  ``boundary potential'' type for solution.
So,  we consider the operator 
\begin{equation*}
\mathcal{V}g(x,t)=\int_0^t\frac{3}{(t-t')}A''\bigg(\frac{x}{(t-t')^{1/3}}\bigg)g(t')dt',
\end{equation*}
for any function $g\in \mathcal{C}_0^{\infty}(\R^+)$, where $A(x)=\displaystyle \frac{1}{2\pi}\int_{\R} e^{i(\xi^3+\xi) x}d\xi$ is the Airy function.

\medskip 
The next Lemmas summarize the principal estimates for the linear Boundary operators $\mathcal{L}$ and $\mathcal{V}.$ 
\begin{lemma}
Let $s\geq 0$ and  $b<\tfrac12$. The following estimates for $\mathcal{L}$ are ensured:
\begin{align*}
&\|\mathcal{L}f(x,t)\|_{\mathcal{C}\big(\mathbb{R}_t^+;\,H^s(\R_x)\big)}\leq c_s\|f\|_{H_0^{(2s+1)/4}(\mathbb{R}^+)},\\
&\|\psi(t)\mathcal{L}f(x,t)\|_{\mathcal{C}\big(\mathbb{R}_x;\,H_0^{(2s+1)/4}(\mathbb{R}_t^+)\big)}\leq c_{\psi, s} \|f\|_{H_0^{(2s+1)/4}(\mathbb{R}^+)},\\
&\|\psi(t)\mathcal{L}f(x,t)\|_{X^{s,b}}\leq c_{\psi, s} \|f\|_{H_0^{(2s+1)/4}(\R^+)}.
\end{align*}
\end{lemma}
\begin{proof}
See lemmas 3.1, 3.2 and 3.3  in  \cite{tzirakis}.
\end{proof}

\begin{lemma}\label{cof}
Let $\kappa \geq 0$, $0< b < \frac12$ and $\frac12 < \alpha < \frac23$.  The following estimates for $\mathcal{V}$ are ensured:
\begin{align*}
&\|\mathcal{V}g(x,t)\|_{\mathcal{C}\big(\R_t;\,H^{\kappa}(\R_x)\big)}\leq c_{\kappa} \|g\|_{H_0^{(\kappa+1)/3}(\R^+)},\\
&\|\psi(t)\mathcal{V}g(x,t)\|_{\mathcal{C}\big(\R_x;\,H_0^{(\kappa+1)/3}(\R_t^+)\big)}\leq c_{\psi, \kappa} \|g\|_{H_0^{(\kappa+1)/3}(\mathbb{R}^+)},\\
&\|\psi(t)\mathcal{V}g(x,t)\|_{Y^{\kappa,b,\alpha}}\leq c_{\psi, \kappa} \|g\|_{H_0^{(\kappa+1)/3}(\R^+)}.
\end{align*}
\end{lemma} 
\begin{proof}
See lemmas 3.5, 3.6 and 3.9 in \cite{Faminskii}.
\end{proof}

\subsection{Non-homogeneous operators} The  Duhamel inhomogeneous solution operator $\mathcal{S}$ associated with Schr\"odinger equation is define by 
\begin{equation*}
\mathcal{S}w(x,t)=-i\int_0^te^{i(t-t')\partial_x^2}w(x,t')dt',
\end{equation*}
so that
\begin{equation}\label{nonlinears}
\begin{cases}
(i\partial_t+\partial_x^2)\mathcal{S}w(x,t) =w(x,t)&\text{for}\quad  (x,t)\in\mathbb{R}\times\mathbb{R},\\
\mathcal{S}w(x,t)=0& \text{for}\quad x\in\mathbb{R}
\end{cases}
\end{equation}
and the corresponding inhomogeneous solution operator $\mathcal{K}$ associated to the KdV equation is given by
\begin{equation*}
\mathcal{K}w(x,t)=\int_0^te^{-(t-t')\partial_x^3}w(x,t')dt',
\end{equation*}
thus we have 
\begin{equation}\label{DK}
\begin{cases}
(\partial_t+\partial_x^3)\mathcal{K}w(x,t) =w(x,t)&\text{for}\quad  (x,t)\in\mathbb{R}\times\mathbb{R},\\
\mathcal{K}w(x,t) =0 & \text{for}\quad x\in\mathbb{R}.
\end{cases}
\end{equation}

\medskip 
The following result is due to Erdogan and Tzirakis and its proof can be seen in \cite{tzirakis}.
\begin{lemma}\label{duhamel-s}
Let $-\frac{1}{2}<d<0$. The following estimates are ensured for the operator $\mathcal{S}$:
\begin{align*}
&\|\psi(t)\mathcal{S}w(x,t)\|_{\mathcal{C}\big(\mathbb{R}_t;\,H^s(\mathbb{R}_x^+)\big)}\lesssim \|w\|_{X^{s,d}}\;\text{with}\;\; s\in \R,\\
&\|\psi(t)\mathcal{S}w(x,t)\|_{\mathcal{C}\big(\mathbb{R}_x;H^{(2s+1)/4}(\mathbb{R}_t)\big)}\lesssim \|w\|_{X^{s,d}}\; \text{with}\;\; -\tfrac{1}{2}< s\leq\tfrac{1}{2},\\
&\|\psi(t)\mathcal{S}w(x,t)\|_{\mathcal{C}\big(\mathbb{R}_x;H^{(2s+1)/4}(\mathbb{R}_t)\big)}\lesssim \|w\|_{X^{\frac{1}{2},\, (2s-1+4d)/4}}\!+\!\|w\|_{X^{s,d}}\;
\text{with}\;\;\tfrac{1}{2}\leq s<\tfrac{5}{2},\\
&\|\psi(t)\mathcal{S}w(x,t)\|_{X^{s,b}}\leq c \|w\|_{X^{s,d}}\; \text{with}\;\; -\tfrac{1}{2}<d\leq0\leq b\leq d+1.
\end{align*}
\end{lemma}

\medskip 
Similar results for KdV equation were obtained by Faminskii in \cite{Faminskii}, which are described in the following lemma. 
\begin{lemma}\label{duhamel-kdv}
Let  $\kappa \ge 0$, $0<b<\frac12$ and  $\frac12<\alpha<\frac23$. Then we have the  following estimates for the operator $\mathcal{K}$:
\begin{align*}
&\|\psi(t)\mathcal{K}w(x,t)\|_{\mathcal{C}\big(\mathbb{R}_t;\,H^{\kappa}(\mathbb{R}_x)\big)}\lesssim \|w\|_{W^{\kappa,b,\alpha}},\\
&\|\psi(t)\mathcal{K}w(x,t)\|_{\mathcal{C}\big(\mathbb{R}_x;\,H^{(\kappa+1)/3}(\mathbb{R}_t)\big)}\lesssim  \|w\|_{W^{\kappa,b,\alpha}},\\
&\|\psi(t)\mathcal{K}w(x,t)\|_{Y^{\kappa,b,\alpha}}\lesssim  \|w\|_{W^{\kappa,b,\alpha}}.
\end{align*}
\end{lemma}

\medskip 
\subsection{\textbf{Nonlinear Estimates}}\label{nonlinearestimates}
Now we summarize the main  nonlinear estimates  that are needed in the proof of well-posedeness results.   

\medskip 
The trilinear estimate corresponding to the cubic-NLS equation, in the context of Borgain spaces with $b< 1/2$, was deduced by 
Erdogan and Tzirakis in \cite{tzirakis}. 
\begin{lemma}\label{trilinear}
Let $s>0$ and $a\leq \min\big\{2s,\frac12\big\}$, then there exists $\epsilon=\epsilon(a, s)>0$ such that 
	\begin{equation*}
	\|u_1u_2\bar{u}_3\|_{X^{s+a,-b}}\lesssim \|u_1\|_{X^{s,b}}\|u_2\|_{X^{s,b}}\|u_3\|_{X^{s,b}}.
	\end{equation*}
for all  $\frac{1}{2}-\epsilon<b<\frac{1}{2}$ and $u_i\in X^{s,b}$, $i=1,2,3$. 
\end{lemma}

\medskip
The bilinear estimate for de nonlinear term of the KdV equation, in the context of Bourgain spaces with $b<\frac{1}{2}$,  was derived by Faminskii in \cite{Faminskii}. 
\begin{lemma}\label{bilinear1}
Let $\kappa >0$, $\frac{7}{16}<b<\frac12$ and $\frac12<\alpha <\frac23$. Then there exists $\epsilon=\epsilon(b,\alpha)>0$ such that 
\begin{equation*}
\big\|\psi_T^2(t)\partial_x (v_1v_2)\big\|_{W^{\kappa,b,\alpha}}\lesssim T^{\epsilon}\|v_1\|_{Y^{\kappa,b,\alpha}}\|v_2\|_{Y^{\kappa,b,\alpha}}.
\end{equation*}
for any $v_i\in Y^{\kappa,b,\alpha}$, $i=1,2$. 
\end{lemma}

\medskip 
The following mixed bilinear estimates can be obtained without major difficulties following the same ideas in \cite{CM}.

\begin{lemma}\label{acoplamento1}
	Let $s,\ \kappa,\ a,\ b$ and $\alpha$ real numbers such that  $\kappa-|s|> \max\big\{2-6b,\,\frac{5}{2}-9a\big\}$, $\frac{7}{18}<2b-\tfrac12 \le a<b$ and $\frac12<\alpha <\frac23$. Then we have
	\begin{equation*}
	\|uv\|_{X^{s,-a}}\lesssim \|u\|_{X^{s,b}}\|v\|_{Y^{\kappa,b,\alpha}}
	\end{equation*}
	for any $u\in X^{s,b}$ and $v\in Y^{\kappa,b}$.
\end{lemma}
\begin{proof}
Follows from Proposition 5.1 in \cite{CM}.
\end{proof}
\begin{lemma}
Let $s,\ \kappa,\ b,$ and $\alpha$ real numbers such that  $\frac38<b<\frac12$, $2s-1-4b<0$, $s,\ \kappa>\frac12$ and $\frac12<\alpha <\frac23$. Then we have
\begin{equation*}
\|uv\|_{X^{\frac{1}{2},\,(2s-1-4b)/4}}\lesssim \|u\|_{X^{s,b}}\|v\|_{Y^{\kappa,b,\alpha}}.
\end{equation*}
for any $u\in X^{s,b}$ and $v\in Y^{\kappa,b,\alpha}$
\end{lemma}
\begin{proof}
The estimate can be obtained following the Case 1 in the proof of Proposition 5.2 in \cite{CM}.
\end{proof}

\begin{lemma}
Let $s,\ \kappa,\ b,\ \alpha$ real numbers such that $s\geq \frac{3}{10}$, $\frac14<b<\frac12$,  $\frac12<\alpha <\frac23$ 
and 
$$0< \kappa \leq \big\{s+3b-1,\, \tfrac{3b}{2},\, \tfrac{2s+3b-7}{2} \big\}.$$ 
Then we have
\begin{equation*}
\|\partial_x(u_1u_2)\|_{W^{\kappa,b,\alpha}}\leq c \|u_1\|_{X^{s,b}}\|u_2\|_{X^{s,b}}.
\end{equation*}
for any $u_i\in X^{s,b}$, $i=1,2$. 
\end{lemma}
\begin{proof}
Follows from Propositions 5.3 and part (a) of Proposition 5.4 in \cite{CM}.
\end{proof}

\subsection{Proof of Theorem \ref{teorema1}}
We only give the sketch of the proof because it is modeled as in \cite{CM}. 

\medskip 
Fix $(s,\kappa)$ as in the hypothesis of Theorem \ref{teorema1} and choice $b(s,\kappa)<\frac12$ and $\frac12<\alpha(s,\kappa)<\frac23$ such that all estimates in Bourgain spaces are valid.

\medskip 
Let $\tilde{u}_0$, $\tilde{v}_0$ extensions of initial data $u_0$, $v_0$ such that
$$\|\tilde{u}_0\|_{H^{s}(\R)}\lesssim\|u_0\|_{H^s(\R^+)}\;\, \text{and}\;\,
\|\tilde{v}_0\|_{H^{\kappa}(\R)}\lesssim \|v_0\|_{H^{\kappa}(\R^+)}.$$
Also consider $\tilde{f}$, $\tilde{g}$  extensions of boundary conditions  $f$ and $g$ such that 
$$\|\tilde{f}\|_{H^{(2s+1)/4}(\R)}\lesssim \|f\|_{H^{(2s+1)/4}(\R^+)}\;\, \text{and}\;\, \|\tilde{g}\|_{H^{(\kappa+1)/3}(\R)}\lesssim \|g\|_{H^{(\kappa+1)/3}(\R^+)}.$$

\medskip
The procedure is to find a fixed point for the operator $\Lambda=(\Lambda_1,\Lambda_2)$, defined by
$$
\begin{cases}
\Lambda_1(u,v)=\psi(t)e^{it\partial_x^2}\tilde{u}_0(x)+\psi(t)\mathcal{S}\big(\alpha \psi_T uv+\beta|\psi_Tu|^2\psi_Tu\big)(x,t)+\psi(t)\mathcal{L}h_1(x,t),\medskip\\ 
\Lambda_2(u,v)=\psi(t)e^{-t\partial_x^3}\tilde{v}_0(x)+\psi(t)\mathcal{K}\big(\gamma\partial_x(|\psi_Tu|^2)-\frac{1}{2}\psi_T^2\partial_xv^2\big)(x,t)+\psi(t)\mathcal{V}h_2(x,t),
\end{cases}
$$
where $h_1=\tilde{h}_1\big|_{\R^+}$  and $h_2=\tilde{h}_2\big|_{\R^+}$ with
\begin{align*}
&\tilde{h}_1(t)=\psi(t)\tilde{f}(t)-\psi(t)e^{it\partial_x^2}\tilde{u}_0(x)-\psi(t)\mathcal{S}\big(\alpha \psi_T u\psi_Tv+\beta|\psi_Tu|^2\psi_Tu\big)(0,t),\\
&\tilde{h}_2(t)=\psi(t)\tilde{g}(t)-\psi(t)e^{-t\partial_x^3}\tilde{v}_0(x)-\psi(t)\mathcal{K}\big(\gamma\partial_x(|\psi_T^2u|^2)
-\tfrac{1}{2}\partial_x((\psi_T^2v)^2)\big)(x,t)+\big)(0,t).
\end{align*}

\medskip
We consider $\Lambda$ defined on the Banach space $Z=Z(s,\kappa)=Z_1(s)\times Z_2(\kappa)$, where
\begin{align*}
&Z_1=\mathcal{C}\big(\R_t;\,H^s(\R_x)\big)\cap \mathcal{C}\big(\R_x;\,H^{(2s+1)/4}(\R_t)\big)\cap X^{s,b},\\
&Z_2=\mathcal{C}\big(\R_t;\,H^{\kappa}(\R_x)\big)\cap \mathcal{C}\big(\R_x;\,H^{(\kappa +1)/3}(\R_t)\big)
\cap Y^{\kappa,b,\alpha},
\end{align*}
and the space $Z$ is equipped wit the norm $\|(u,v)\|_Z=\|u\|_{X^{s,b}} + \|v\|_{Y^{\kappa,b,\alpha}}.$

\medskip 
By using the estimates of linear groups, boundary  operators and all the nonlinear estimates given in this section combined with Lemmas \ref{Xsb-estimates-GTV} and \ref{faminski1} we have that there exists a constant $c=c(s,\kappa)>0$ and 
$0<\epsilon =\epsilon(s,\kappa)\ll 1$  such that
\begin{align}
&\|\Lambda_1(u,v)\|_{Z_1}\leq  c\big(\|u_0\|_{H^s(\R^+)}+\|f\|_{H^{2s+1)/4}(\R^+)}+T^{\epsilon}(\|u\|_{X^{s,b}}\|v\|_{Y^{\kappa,b,\alpha}}
+\|u\|_{X^{s,b}}^3\big),\label{cont1}\\ 
&\|\Lambda_2(u,v)\|_{Z_2}\leq  c(\|v_0\|_{H^{\kappa}(\R^+)}+\|g\|_{H^{(\kappa+1)/3}(\R^+)}+T^{\epsilon}(\|u\|_{X^{s,b}}^2+\| v\|_{Y^{\kappa,b,\alpha}})\label{cont2}
\end{align}
and
\begin{equation}\label{cont3}
\|\Lambda(u_1,v_1)-\Lambda(u_2,v_2)\|_{Z}\le cT^{\epsilon}\Phi(u_1, v_1; u_2,v_2)
\big( \|u_1-u_2\|_{X^{s,b}}  + \|v_1-v_2\|_{Y^{\kappa,b,\alpha}}\big),
\end{equation}
where
$$
\Phi=\|u_1\|_{X^{s,b}} + 2\|u_2\|_{X^{s,b}} + 2\|v_1\|_{Y^{\kappa,b,\alpha}} + \|v_2\|_{Y^{\kappa,b,\alpha}}
+\|u_1\|_{X^{s,b}}^2+\|u_2\|_{X^{s,b}}^2.
$$

\medskip
Hence, considering initial-boundary data on a ball of $Z$ we have from \eqref{cont1}, \eqref{cont2} and \eqref{cont3} that there exists a positive time $T$, depending of the norm of the data, such that  $\Lambda$ defines a contraction an the rest of the proof follows by standard arguments. 

\section{\textbf{Long-time behavior in  $\boldsymbol{H^1(\R^+)\times H^1(\R^+)}$}}\label{4}

Now we give the proof for Theorem \ref{Th-global-theory}. Before starting the proof we recall that we are under hypotheses of homogeneous boundary condition (i.e. $f=g=0$).

\subsection{Proof of Theorem \ref{Th-global-theory}}
The prof of global extension in time of the local solutions is very simples, then we only give an sketch. As seen in \eqref{i10}, the $\displaystyle \|u(\cdot, t)\|_{L^2(\R^+)}$ is preserved when $f=0$, so we only need to obtain an a priori estimate for the continuous function  
\begin{equation}\label{function-varphi}
\varphi^+(t):= \|v(\cdot, t)\|^2_{L^2(\R^+)} + \|u_x(\cdot, t)\|^2_{L^2(\R^+)} + \|v_x(\cdot, t)\|^2_{L^2(\R^+)},
\end{equation}
for all $<0 < t < T_{\ast}$, where $ T_{\ast}$ denotes  the maximal time of existence of the solution. 

\medskip 
From \eqref{i20} we have 
\begin{equation}\label{long-time-right-proof-a}
\int_0^{+\infty}\Big(\frac{\alpha}{\gamma}v^2+2\,\text{Im}\big(u\bar{u}_x)\Big)dx\le \mathcal{Q}^+_0
\end{equation}
and from \eqref{i30} with $\alpha \gamma >0$ we have 
\begin{equation}\label{long-time-right-proof-b}
\int_0^{+\infty}\Big(\alpha v|u|^2-\frac{\alpha}{6\gamma}v^3+\frac{\beta}{2}|u|^4+\frac{\alpha}{2\gamma}v_x^2+|u_x|^2\Big)dx\le \mathcal{E}^+_0.
\end{equation}
Since $\alpha\gamma >0$, using \eqref{long-time-right-proof-a} and \eqref{long-time-right-proof-b}, we derive 
in the similar way as in the proof of Theorem 1.2 in \cite{CL} the following estimate:
\begin{equation}\label{long-time-right-proof-c}
\|v\|_{H^1(\R^+)}^2+\|u\|_{H^1(\R^+)}^2\leq \Psi(\|u_0\|_{H^1(\R^+)},\|v_0\|_{H^1(\R^+)}),
\end{equation}
where $\Psi$ is a continuous function depending only on $\|u_0\|_{H^1(\R^+)}$ and $\|v_0\|_{H^1(\R^+)}$.
So, the a priori estimate in \eqref{long-time-right-proof-c} allows to extend the solution globally in time.  

\medskip
Now we proceed  with the second part of the theorem. So, we consider now that data $u_0\in H^1(\R^+)\cap\in L^2\big(\R^+,xdx\big)$.

\medskip 
Multiplying the first equation of the model  by $x\bar{u}$, integrating on $\R+$ and taking the imaginary part we obtain the relation
\begin{equation}\label{long-time-right-proof-d}
\text{Re}\int_0^{+\infty}\!\!\!x\bar{u}u_tdx + \text{Im}\int_0^{+\infty}\!\!\!x\bar{u}u_{xx}dx
=\alpha \text{Im}\int_0^{+\infty}\!\!\!x|u|^2vdx+ \beta \text{Im}\int_0^{+\infty}\!\!\!x|u|^4dx.
\end{equation}
Then, combining \eqref{long-time-right-proof-d} and \eqref{i20} we have 
\begin{equation}\label{long-time-right-proof-e}
\begin{split}
\frac{d}{dt}\int_0^{+\infty}\!\!\!x|u|^2dx &= -2\,\text{Im}\int_0^{+\infty}\!\!\!u\bar{u}_xdx\\
&=\frac{\alpha}{\gamma}\int_0^{+\infty}\!\!\!v^2-\mathcal{Q}^+_0 + \int_0^t\Big(2|u_x(0,s)|^2 +\frac{\alpha}{\gamma}v_x^2(0,s)\Big)ds\\
&>-\mathcal{Q}^+_0,
\end{split}
\end{equation}
since that $\alpha \gamma >0$. Then, if $\mathcal{Q}^+_0 < 0$ the assertion in \eqref{Th-global-theory-b} follows by integration in time of the last inequality and the proof is complete.

\section{\textbf{Long-time behavior in  $\boldsymbol{H^1(\R^-)\times H^1(\R^-)}$}}\label{5}
In this section we exhibit the proof of Theorem \ref{Th-grow-time-left}. For the second assertion of theorem we will use 
the following global virial identity associated to the model.  

\subsection{Global virial identitiy} 
	
\begin{proposition}[Virial Identity]\label{lemaviriel}
Let $(u(\cdot,t),v(\cdot,t))\in \mathcal{C}\big([0,T];\; H^1(\R^-)\times H^1(\R^-)\big)$  the solution of the IBVP \eqref{SK} provided by Theorem \ref{teorema1left}, under homogeneous boundary conditions ($f=g=h=0$), in the space associated to an enough regular data such that $(u_0,v_0)\in L^2(\R^-, |x|dx)\times L^2(\R^-,|x|dx)$. Then, the function 
\begin{equation*}
\eta(t)=\int_{-\infty}^0|x|^2|u|^2dx-\frac{2\alpha}{\gamma}\int_0^t\int_{-\infty}^0xv^2 + \int_0^t\int_{-\infty}^0x|u|^2
\end{equation*}
belongs to the class $\mathcal{C}^2\big([0,T];\; \R\big)$. Furthermore, we have 
\begin{equation}\label{viriel1}
\frac{d}{dt}\int_{-\infty}^0|x|^2|u|^2dx=4\emph{Im}\int_{-\infty}^0x\bar{u}u_xdx
\end{equation}
and
\begin{multline}\label{viriel2}
\eta''(t)=8\int_{-\infty}^0\!\!|u_x|^2dx+\frac{6\alpha}{\gamma}\int_{-\infty}^0\!\!v_x^2dx\\ 
+ 2\beta\int_{-\infty}^0\!\!|u|^4dx-\frac{4\alpha}{3\gamma}\int_{-\infty}^0\!\!v^3dx + 4\alpha\int_{-\infty}^0\!\!v|u|^2dx -2\emph{Im}\;\int_{-\infty}^0\!\!u\bar{u}_x.
\end{multline}
\end{proposition}

\begin{proof}
Formally, to obtain \eqref{viriel1} by standard procedure we multiply the  first equation of \eqref{SKe}
by $2\,\bar{u}x^2$ and we take the imaginary part of the integration on  $(-\infty, 0)$ to get
\begin{equation*}
2\text{Re}\int_{-\infty}^0x^2\bar{u}u_tdx+2\text{Im}\int_{-\infty}^0x^2\bar{u}u_{xx}dx=
2\text{Im}\int_{-\infty}^0x^2|u|^2(v+\beta|u|^2)dx,
\end{equation*}
consequently
\begin{equation*}
\frac{d}{dt}\int_{-\infty}^0x^2|u|^2dx=4\text{Im}\int_{-\infty}^0x \bar {u}u_xdx,
\end{equation*}
as claimed in \eqref{viriel1}. 

\medskip 
 Now we proceed to obtain \eqref{viriel2}. After some integration by parts we have
 \begin{equation}
 \begin{split}
 \frac{d}{dt^2}\int_{-\infty}^0x^2|u|^2dx&=4\text{Im}\left(\int_{-\infty}^0x \bar{u}_tu_xdx+
 \int_{-\infty}^0x \bar{u}u_{xt}dx\right)\\
 &=4\text{Im}\left(\int_{-\infty}^0x\bar{u}_tu_xdx-\int_{-\infty}^0x\bar{u}_xu_{t}dx-\int_{-\infty}^0 \bar{u}u_{t}dx\right)\\
 &=4\text{Im}\left(2i\text{Im}\int_{-\infty}^0x\bar{u}_t u_xdx-\int_{-\infty}^0 \bar{u}u_{t}dx \right)\\
 &=8\text{Im}\int_{-\infty}^0x\bar{u}_t u_xdx-4\text{Im}\int_{-\infty}^0\bar{u}u_{t}dx\\
 &:=8A-4B.
 \end{split}
 \end{equation}
 
 Using the fist equation of \eqref{SKe} we have
 \begin{equation}
 \begin{split}
 A&=\text{Im}\left(-i \int_{-\infty}^0x\bar{u}_{xx}u_xdx+i\alpha\int_{-\infty}^0x\bar{u}vu_xdx+i\beta \int_{-\infty}^0x|u|^2\bar{u}u_xdx \right)\\
 &= -\text{Re} \int_{-\infty}^0x\bar{u}_{xx}u_xdx+\alpha\,\text{Re}\int_{-\infty}^0x\bar{u}vu_xdx+\beta\,\text{Re} \int_{-\infty}^0x|u|^2\bar{u}u_xdx\\
 &:=A_1+A_2+A_3.
 \end{split}
 \end{equation}
 
 Now, integrating by parts we get
 \begin{equation}\label{A1}
 A_1= -\frac{1}{2}\int_{-\infty}^0x\big(\bar{u}_{xx}u_x + u_{xx}\bar{u}_x \big)dx  
 = -\frac{1}{2}\int_{-\infty}^0x\big(|u_x|^2\big)_xdx = \frac{1}{2}\int_{-\infty}^0|u_x|^2dx.
 \end{equation}
 For $A_2$ we obtain
 \begin{equation}
 A_2=\frac{\alpha}{2}\int_{-\infty}^0xv\big(\bar{u}u_x + \bar{u}_xu\big)dx=
 \frac{\alpha}{2}\int_{-\infty}^0xv\big(|u|^2\big)_xdx.
 \end{equation}
 
 Now, using the second equation of \eqref{SKe} we have
 \begin{equation*}
 \begin{split}
 \frac{2\gamma}{\alpha}A_2&=\int_{-\infty}^0xvv_tdx + \int_{-\infty}^0xvv_{xxx}dx +\frac{1}{2}\int_{-\infty}^0xv(v^2)_xdx\\
 &=\frac{1}{2}\frac{d}{dt}\int_{-\infty}^0\!\!\!xv^2dx-\int_{-\infty}^0\!\!\!vv_{xx}dx -\int_{-\infty}^0\!\!\!xv_xv_{xx}dx
 -\frac{1}{2}\int_{-\infty}^0\!\!\!v^3dx-\frac{1}{2}\int_{-\infty}^0\!\!\!xv_xv^2dx\\
 &=\frac{1}{2}\frac{d}{dt}\int_{-\infty}^0\!\!\!xv^2dx + \int_{-\infty}^0\!\!\!v_x^2dx 
 -\frac{1}{2}\int_{-\infty}^0\!\!\!x(v_x^2)_x -\frac{1}{2}\int_{-\infty}^0\!\!\!v^3dx 
 -\frac{1}{6}\int_{-\infty}^0\!\!\!x(v^3)_xdx\\
 &=\frac{1}{2}\frac{d}{dt}\int_{-\infty}^0\!\!\!xv^2dx +\frac{3}{2}\int_{-\infty}^0\!\!\!v_x^2dx-\frac{1}{3}\int_{-\infty}^0\!\!\!v^3dx;
\end{split}
 \end{equation*}
so we have 
\begin{equation}\label{A2}
A_2=\frac{\alpha}{4\gamma}\frac{d}{dt}\int_{-\infty}^0\!\!\!xv^2dx +\frac{3\alpha}{4\gamma}\int_{-\infty}^0\!\!\!v_x^2dx-\frac{\alpha}{6\gamma}\int_{-\infty}^0\!\!\!v^3dx
\end{equation}
 
 For $A_3$ we have 
 \begin{equation}\label{A3}
 \begin{split}
 A_3&=\frac{\beta}{2}\int_{-\infty}^0x|u|^2(\bar{u}u_x+u\bar{u}_x)dx\\
 &=\frac{\beta}{2}\int_{-\infty}^0x\frac{1}{2}(|u|^4)_xdx\\
 & =-\frac{\beta}{4}\int_{-\infty}^0|u|^4dx. 
 \end{split}
 \end{equation}
 
 In order to calculate $B$ we use the first equation to get
 \begin{equation}\label{B}
 \begin{split}
 B&=\text{Im}\left(i\int_{-\infty}^0u_{xx}\bar{u}dx-i\alpha\int_{-\infty}^0|u|^2v-\beta i\int_{-\infty}^0|u|^4dx \right)\\
 &=-\int_{-\infty}^0|u_x|^2dx-\alpha\int_{-\infty}^0|u|^2v-\beta\int_{-\infty}^0|u|^4dx.
 \end{split}
 \end{equation}
 
 On the other hand, similar to calculations in  \eqref{long-time-right-proof-d} we have
 
 \begin{equation}\label{virial-final}
 \frac{d}{dt}\int_{-\infty}^0x|u|^2dx = -2\,\text{Im}\int_{-\infty}^0u\bar{u}_xdx.
 \end{equation}
 
Finally,  the virial identity \eqref{viriel2} follows by combining the estimates \eqref{A1}, \eqref{A2}, \eqref{A3}, \eqref{B} and \eqref{virial-final}. 
 
In all the computations above we assume regularity enough for the solution. To justify the calculus in the functional spaces 
considered in the hypotheses of Proposition \ref{lemaviriel}, can be used the regularization argument used in \cite{Cazenavebook} in the context of Nonlinear Schr\"odinger equation.
\end{proof}

\subsection{Proof of Theorem \ref{Th-grow-time-left}}
We consider system \eqref{SKe} with $f=g=h=0$, then from \eqref{i20} and \eqref{i30} we have

\begin{equation}\label{grow-time-left-proof-a}
2\text{Im} \int_{-\infty}^0u\bar{u}_xdx = \mathcal{Q}_0^- + 2\int_0^t|u_x(0,s)|^2ds - \frac{\alpha}{\gamma}\int_{-\infty}^0v^2dx
\end{equation}

and

\begin{multline}\label{grow-time-left-proof-b}
8\int_{-\infty}^0|u_x|^2dx +\frac{4\alpha}{\gamma}\int_{-\infty}^0v_x^2dx + 4\beta\int_{-\infty}^0|u|^4dx -\frac{4\alpha}{3\gamma}\int_{-\infty}^0v^3dx\\
+ 8\alpha\int_{-\infty}^0v|u|^2dx =8\mathcal{E}^-_0 + \frac{4\alpha}{\gamma}\int_0^tv_{xx}^2(0,s)ds.
\end{multline}

\smallskip 
\noindent \textbf{Proof of the  assertion (a).} From \eqref{virial-final} and \eqref{grow-time-left-proof-a} we get 
\begin{equation}\label{grow-time-left-proof-a1}
\begin{split}
\frac{d}{dt}\int_{-\infty}^0|x| |u|^2dx & = 2\,\text{Im}\int_{-\infty}^0u\bar{u}_xdx\\
 &=\mathcal{Q}_0^- + 2\int_0^t|u_x(0,s)|^2ds - \frac{\alpha}{\gamma}\int_{-\infty}^0v^2dx.
\end{split}
\end{equation}
So, since $\alpha/\gamma <0$ we have 
$$\frac{d}{dt}\int_{-\infty}^0|x| |u|^2dx \ge \mathcal{Q}_0^-.$$
Hence
$$\int_{-\infty}^0|x| |u|^2dx \ge \mathcal{Q}_0^-t + \int_{-\infty}^0|x| |u_0|^2dx.$$

\smallskip 
\noindent \textbf{Proof of the assertion (b).} Combining \eqref{viriel2} with \eqref{grow-time-left-proof-a} and \eqref{grow-time-left-proof-b}  we obtain 

\begin{multline}\label{grow-time-left-proof-c}
\eta''(t)=8\mathcal{E}_0^- + \frac{4\alpha}{\gamma}\int_0^tv_{xx}^2(0,s)ds + \frac{2\alpha}{\gamma}\int_{-\infty}^0v_x^2dx -2\beta\int_{-\infty}^0|u|^4dx\\
-4\alpha\int_{-\infty}^0v|u|^2dx  -\mathcal{Q}_0^- -2\int_0^t|u_x(0,s)|^2ds + \frac{\alpha}{\gamma}\int_{-\infty}^0v^2dx.
\end{multline}

The hypothesis  $\alpha/\gamma < 0$ yield us the inequality 
$$\eta''(t) < 8\mathcal{E}_0^- - \mathcal{Q}_0^- -2\beta\int_{-\infty}^0|u|^4dx+ \frac{\alpha}{\gamma}\int_{-\infty}^0v^2dx  -4\alpha\int_{-\infty}^0v|u|^2dx;$$
so Young's inequality implies that 
$$\eta''(t) < 8\mathcal{E}_0^- - \mathcal{Q}_0^- + (4|\alpha \gamma| -2\beta)\int_{-\infty}^0|u|^4dx.$$
Hence, for $\beta \ge  2|\alpha \gamma|$ we have the uniform bond 

\begin{equation}\label{grow-time-left-proof-d}
\eta''(t) < 8\mathcal{E}_0^- - \mathcal{Q}_0^-,\quad 0 \le t < T_{\ast},
\end{equation}
where $ T_{\ast}$ denotes  the maximal time of existence of the solution.  

Now  we  assume that  $T_{\ast}=+\infty$. Then, after two integrations in time we have from \eqref{grow-time-left-proof-d} the following estimate
\begin{multline*}
\int_{-\infty}^{0}x^2|u|^2dx + \frac{2\alpha}{\gamma}\int_0^t\int_{-\infty}^0xv^2dxds - \int_0^t\int_{-\infty}^0x|u|^2dxds\\
<\frac{ 8\mathcal{E}_0^- - \mathcal{Q}_0^-}{2}t^2 + \eta'(0)t + \eta(0),
\end{multline*}
where 

\begin{itemize}
\item $\displaystyle \eta'(0)= 4\text{Im}\int_{-\infty}^0x \bar{u}_0\partial_xu_0dx + \frac{2\alpha}{\gamma}\int_{-\infty}^0|x|v_0^2dx-\int_{-\infty}^0|x||u_0|^2dx,$ \medskip 
\item $\displaystyle \eta(0)= \int_{-\infty}^0x^2|u_0|^2dx.$
\end{itemize}
Thus, 
\begin{multline*}
0 < \int_{-\infty}^{0}x^2|u|^2dx < \frac{ 8\mathcal{E}_0^- - \mathcal{Q}_0^-}{2}t^2 + \eta'(0)t + \eta(0)\\
- \frac{2\alpha}{\gamma}\int_0^t\||x|^{1/2}v(\cdot, s)\|^2_{L^2(\R^-)}ds+ \int_0^t\||x|^{1/2}u(\cdot, s)\|^2_{L^2(\R^-)}ds,
\end{multline*}
which implies 
\begin{equation}\label{grow-time-left-proof-e}
\frac{\mathcal{Q}_0^- - 8\mathcal{E}_0^-}{2}t^2 - \eta'(0)t - \eta(0) < \int_0^tP_{u, v}(s)ds,
\end{equation}
where
$$P_{u,v}(s):= \|\,|x|^{1/2}u(\cdot, s)\|^2_{L^2(\R^-)} + 2\big|\tfrac{\alpha}{\gamma}\big|\, \|\,|x|^{1/2}v(\cdot, s)\|^2_{L^2(\R^-)}$$

Now we note that if $\mathcal{Q}_0^- >  8\mathcal{E}_0^-$ we should have, for all time   $t>0$, the inequality

\begin{equation}\label{grow-time-left-proof-f}
\sup\limits_{s\in [0,\,t]}P_{u, v}(s) > \frac{\mathcal{Q}_0^- - 8\mathcal{E}_0^-}{2}t - \eta'(0)- \frac{\eta(0)}{t},
\end{equation}

because otherwise we would have $P_{u, v}(s) \le \frac{\mathcal{Q}_0^- - 8\mathcal{E}_0^-}{2}t - \eta'(0)- \frac{\eta(0)}{t}$
for all $0\le s \le t$. 

Finally, from \eqref{grow-time-left-proof-f} it follows that 

$$\lim\limits_{t\to +\infty} \frac{1}{t^{1-}}\|P_{u,v}\|_{L^{\infty}([0,\,t])}=+\infty;$$

consequently 

$$\lim\limits_{t\to +\infty}\Big( \|\,|x|^{1/2}u(\cdot, t)\|^2_{L^2(\R^-)} + \|\,|x|^{1/2}v(\cdot, t)\|^2_{L^2(\R^-)}\Big) =+\infty$$

\medskip 
and the proof is finished.

\medskip 
\section{\textbf{Appendix}}\label{laws-deduction}

Here we will derive the conserved quantities \eqref{i1}, \eqref{i2} and \eqref{i3}. 

\subsection{Derivation of the law for the mass}\label{mass-homo-regularization}

The identity \eqref{i1} can be obtained formally by multiplying the equation by $\bar{u}$ and integrating by parts. As the solution obtained in Theorem \ref{teorema1} solves the IBVP in the distributional sense, we need to regularize the solutions for justify the calculus. For the convenience to the reader, we present here the argument of regularization on the case of homogeneous boundary conditions $f=g=0$.

Let $(u,v)$ the solution for IBVP \eqref{SK} in the space:
$$\mathcal{C}([0,T^*);\,H^1(\R^+))\cap \mathcal{C}(\R^+;\,H^{3/4}(0,T^*))\times 
\mathcal{C}([0,T^*);\,H^1(\R^+))\cap \mathcal{C}(\R^+;\,H^{2/3}(0,T^*))$$
associated to the initial-boundary data 
$$(u_0,v_0)\in H^1(\R^+)\times H^1(\R^+) \quad \text{and}\quad (f,g)\in H^{3/4}(\R^+)\times H^{2/3}(\R^+).$$

Let $\tilde{u}$ and $\tilde{v}$ nice extensions of $u$ and $v$ for all $\R^2$ and let $\theta(t)$ a non-negative smooth function supported on $[-\frac{1}{2},-1]$ with $\displaystyle \int_{\R} \theta(t) dt=1$. 
For $\delta$, $\epsilon >0$ we define  $\theta_{\delta}(t)=\delta^{-1}\theta(\delta^{-1}t)$ and then we consider
\begin{equation}
u_{\delta,\epsilon}(x,t)=\int \int \tilde{u}(y,s) \theta_{\epsilon}(x-y) \theta_{\delta}(t-s)dy ds,
\end{equation} 
\begin{equation}
v_{\delta,\epsilon}(x,t)=\int \int \tilde{v}(y,s) \theta_{\epsilon}(x-y) \theta_{\delta}(t-s)dy ds.
\end{equation} 

Note that the functions are supported in $(-\frac{\delta}{2},T-\delta)\times (-\frac{\delta}{2},+\infty)$ and 
since $(u(t),v(t))\in H^1(\R^+)\times H^1(\R^+)$ both  integrals are well-defined in the classical sense.

We start out by making a few observations about these functions. Fix $T_1< T_1^*$ and $T_2< T_2^*$ and consider $\epsilon\leq \min\{T_1^*-T_1, T_2^*-T_2\}$. Suppose that $(x,t)\in \R^+\times (0,T_i)$, then the integration on the definition of   
 $u_{\delta,\epsilon}$ and $v_{\delta,\epsilon}$ take places on 
 $$\Big\{y:\,\frac{\epsilon}{2}+x\leq y< +\infty\Big\}\quad  \text{and}\quad  \Big\{s:\,\frac{\delta}{2}+t\leq s\leq t+\delta\Big\}.$$ 
 Thus since  $u_{\delta,\epsilon}$ and $v_{\delta,\epsilon}$ are in $C^{\infty}(\R^2)$, we see that
 \begin{equation}\label{regular}
 \begin{cases}
 i\partial_tu_{\delta,\epsilon}+\partial_x^2 u_{\delta,\epsilon}=\alpha u_{\delta,\epsilon}v_{\delta,\epsilon}+\beta u_{\delta,\epsilon}|u_{\delta,\epsilon}|^2,&  (x,t)\in\mathbb{R}^+\times(0,T_1),\medskip\\
 \partial_tv_{\delta,\epsilon}+\partial_{x}^3v_{\delta,\epsilon}+ v_{\delta,\epsilon}\partial_xv_{\delta,\epsilon}=\gamma\partial_x(|u_{\delta,\epsilon}|^2),&  (x,t)\in\mathbb{R}^+\times(0,T_2),
 \end{cases}
 \end{equation}
 in the classical sense. Notice also that 
 \begin{equation}
 u_{\delta,\epsilon}\in \mathcal{C}([0,T_1^*);\, H^1(\R^+))\cap \mathcal{C}(\R^+;\,H^{3/4}(0,T_1^*))
 \end{equation}
 and 
 \begin{equation}
 v_{\delta,\epsilon}\in \mathcal{C}([0,T_2^*);\,H^1(\R^+))\cap \mathcal{C}(\R^+;\, H^{2/3}(0,T_2^*))
 \end{equation}
 uniformly in $(\epsilon,\delta)$.
 
 For $\epsilon$ and $\delta$ fixed we see that $\partial_tu_{\delta,\epsilon}$, $\partial_tv_{\delta,\epsilon}, \partial_x^ku_{\delta,\epsilon}, \partial_x^kv_{\delta,\epsilon}\in L^2(\R_x^+)$, with the bounds
 \begin{align}
 &\|\partial_t u_{\delta,\epsilon}\|_{L^2(\R^+)}\lesssim  \frac{1}{\delta} \|u\|_{L^2(\R^+)},\quad \;\; \|\partial_t v_{\delta,\epsilon}\|_{L^2(\R^+)}\lesssim  \frac{1}{\delta} \|v\|_{L^2(\R^+)},\\
 &\|\partial_x^{k} u_{\delta,\epsilon}\|_{L^2(\R^+)}\lesssim  \frac{1}{\epsilon^{k}} \|u\|_{L^2(\R^+)},\quad
 \|\partial_x^{k} v_{\delta,\epsilon}\|_{L^2(\R^+)}\lesssim  \frac{1}{\epsilon^{k}} \|v\|_{L^2(\R^+)}.
 \end{align}
 
 Multiplying the first equation of \eqref{regular} by $\bar{u}_{\delta,\epsilon}$ and integrating by parts we obtain 
 \begin{equation}\label{i1regular}
\mathcal{M}(t)=\int_0^{+\infty}|u_{\delta,\epsilon}(x,t)|^2dx
=\mathcal{M}(0)+2\int_0^t\text{Im}(\partial_xu_{\delta,\epsilon}(0,s)\bar{u}_{\delta,\epsilon}(0,s))ds.
 \end{equation}

From the definition of $u_{\delta,\epsilon}(t)$ we see that 
 \begin{equation}\label{ap1}
 \sup_{t\in (0,T_1^*)}\|u_{\delta,\epsilon}(\cdot, t)-u(\cdot, t)\|_{L^2(\R^+)}\rightarrow 0,\ \text{when}\ \epsilon,\delta \rightarrow 0\; 
 \end{equation}
 and
\begin{equation}\label{ap1-a}
\sup_{x\in \R^+}\|u_{\delta,\epsilon}(x,\cdot)-u(x,\cdot)\|_{L^2(0,t)}\rightarrow 0,\ \text{when}\ \epsilon,\delta \rightarrow 0.
\end{equation}

On the other hand,  Cauchy-Schwarz inequality implies that
\begin{equation}\label{ap0}
2\int_0^t\text{Im}(\partial_xu_{\delta,\epsilon}(0,s)\bar{u}_{\delta,\epsilon}(0,s))ds \leq \|u_{\delta,\epsilon}(0,s)\|_{L^2(0,T)}\|\partial_xu_{\delta,\epsilon}(0,s)\|_{L^2(0,T)},
\end{equation}
so using the fact $f=0$ and \eqref{ap1-a} we have that
\begin{equation}\label{ap01}
\|u_{\delta,\epsilon}(0,s)\|_{L^2(0,T)}\rightarrow 0,\; \text{when}\ \epsilon,\delta \rightarrow 0. 
\end{equation}
Now we estimate $\|\partial_xu_{\delta,\epsilon}(0,s)\|_{L^2(0,T)}$. By the mean value Theorem, given $L>0$, there exists $0<x_1<L$ such that
\begin{equation}\label{ap2}
\partial_xu_{\delta,\epsilon}(x_1,t)=\frac{1}{L}\left[u_{\delta,\epsilon}(L,t)-u_{\delta,\epsilon}(0,t)\right].
\end{equation}
Again by the mean value Theorem, there exists $x_2$ with $0<x_2<x_1$ such that
\begin{equation}\label{ap3}
\partial_xu_{\delta,\epsilon}(x_1,t)-\partial_xu_{\delta,\epsilon}(0,t)=x_1\partial_x^2u_{\delta,\epsilon}(x_2,t).
\end{equation}
Subtracting \eqref{ap2} and \eqref{ap3} give us
\begin{equation}\label{ap4}
\|\partial_xu_{\delta,\epsilon}(0,\cdot)\|_{L^2(0,T)}\leq L\sup_{0\leq y\leq L}\|\partial_x^2u_{\delta,\epsilon}(y,\cdot)\|_{L^2(0,T)} +L^{-1}\sup_{0\leq y\leq L}\|u_{\delta,\epsilon}(y,\cdot)\|_{L^2(0,T)}.
\end{equation}
Now we fix $L$ such that 
\begin{equation}
\sup_{0\leq y\leq L}\|u(y,\cdot)\|_{L^2(0,T)}+\|v(y,\cdot)\|_{L^2(0,T)}\leq 1
\end{equation}
and we estimate each term in the \eqref{ap4}. For the second term we have
\begin{equation}
L^{-1}\sup_{0\leq y\leq L}\|u_{\delta,\epsilon}(y,\cdot)\|_{L^2(0,T)} \leq L^{-1}\sup_{\delta\leq y\leq -2\delta +L}\|u(y,\cdot)\|_{L^2(0,T)}\leq L^{-1}.
\end{equation}
To estimate the first term of \eqref{ap4} we use the relation
$$i\partial_tu_{\delta,\epsilon}+\partial_x^2 u_{\delta,\epsilon}=\alpha u_{\delta,\epsilon}v_{\delta,\epsilon}+\beta u_{\delta,\epsilon}|u_{\delta,\epsilon}|^2.$$
Simple calculations show  that
\begin{equation}\label{ap5}
\|\partial_tu_{\delta,\epsilon}\|_{L^2(0,T)}\leq \delta^{-1}\sup_{\epsilon\leq y\leq L+2\epsilon}\|u(x,\cdot)\|_{L^2(0,T)}\leq \delta^{-1}.
\end{equation}
By Sobolev immersion we have that
\begin{equation}\label{ap6}
\begin{split}
\sup_{\delta\leq y\leq -2\delta +L}\|u_{\delta,\epsilon}|u_{\delta,\epsilon}|^2\|_{L^2(0,T)}
&\leq \sup_{\delta\leq y\leq -2\delta +L} \|u\|^2_{L^{\infty}(0,T)}\|u_{\delta,\epsilon}\|_{L^{2}(0,T)}\\
&\leq \sup_{\delta\leq y\leq -2\delta +L} \|u_{\delta,\epsilon}\|_{H^1(0,T)}^2\|u_{\delta,\epsilon}\|_{L^{2}(0,T)}\\
&\leq 1+\delta^{-2}
\end{split}
\end{equation}
and
\begin{equation}\label{ap7}
\begin{split}
\sup_{\delta\leq y\leq -2\delta +L}\|u_{\delta,\epsilon}v_{\delta,\epsilon}\|_{L^2(0,T)}&\leq \sup_{\delta\leq y\leq -2\delta +L} \|v_{\delta,\epsilon}\|_{L^2(0,T)}\|u_{\delta,\epsilon}\|_{L^{\infty}(0,T)}\\
&\leq \sup_{\delta\leq y\leq -2\delta +L} \|v_{\delta,\epsilon}\|_{L^2(0,T)}\|u_{\delta,\epsilon}\|_{H^1(0,T)}\\
&\leq 1+\delta^{-1}.
\end{split}
\end{equation}
Combining \eqref{ap5}, \eqref{ap6} and \eqref{ap7} we get
\begin{equation}\label{ap8}
L\sup_{0\leq y\leq L}\|\partial_x^2u_{\delta,\epsilon}(y,\cdot)\|_{L^2(0,T)}\leq L(2\delta^{-1}+2+\delta^{-2}).
\end{equation}
Hence, using \eqref{ap0}, \eqref{ap01} and \eqref{ap8} for fixed $\delta$ we obtain
\begin{equation}
2\int_0^t\text{Im}(\partial_xu_{\delta,\epsilon}(0,s)\bar{u}_{\delta,\epsilon}(0,s))ds \rightarrow 0,\; \text{when}\;\epsilon\rightarrow 0.
\end{equation}
Send $\epsilon\rightarrow 0$ in \eqref{i1regular} we obtain
\begin{equation}
\int_0^{+\infty}|u_{\delta}(x,t)|^2dx=\int_0^{+\infty}|u_{\delta}(x,0)|^2dx.
\end{equation}
Finally,  sending $\delta \rightarrow 0$ we get
\begin{equation}
\mathcal{M}(t)=\int_0^{+\infty}|u(x,t)|^2dx=\mathcal{M}(0).
\end{equation}

\subsection{Formal derivation for the moment functional on the positive half-line}
We begin by multiplying the second equation of \eqref{SK} by $v$ and integrating in $(0,+\infty)$ we have that
\begin{equation}
\frac{\alpha}{2\gamma}\frac{d}{dt} \int_{0}^{+\infty}v^2dx+ \frac{\alpha}{\gamma}\int_{0}^{+\infty}v_{xxx}vdx+
\frac{\alpha}{\gamma}\int_{0}^{+\infty}v^2v_xdx=\alpha\int_{0}^{+\infty} (|u|^2)_xvdx
\end{equation}

By directly integration we get
\begin{equation}
\int_0^{+\infty}v_{xxx}v dx=-\int_0^{+\infty}v_{xx}v_xdx-v_{xx}(0,t)v(0,t)=\frac12 v_x^2(0,t)-v_{xx}(0,t)v(0,t)
\end{equation}
and
\begin{equation}
\int_{0}^{+\infty} v^2v_xdx=\int_0^{+\infty}\frac{d}{dx}\frac{v^3}{3} dx=-\frac{v^3}{3}(0,t). 
\end{equation}
Thus 
\begin{equation}\label{eq-moment-initial}
\frac{\alpha}{\gamma}\frac{d}{dt} \int_{0}^{+\infty}v^2dx + \mathcal{Q}^{v}(t)=2\alpha\int_{0}^{+\infty} (|u|^2)_xvdx ,
\end{equation}
with $ \mathcal{Q}^{v}(t)$ defined in \eqref{i2-b}.

On the other hand, 
\begin{equation}\label{eq1}
\begin{split}
\alpha \int_{0}^{+\infty} (|u|^2)_xvdx=2 \alpha\, \text{Re}\int_0^{+\infty}uv\bar{u}_x dx
\end{split}
\end{equation} 
and also, from the first equation of \eqref{SK}, we have 
\begin{equation}\label{eq1-a}
\alpha\,\text{Re}\int_0^{+\infty}uv\bar{u}_x dx=\text{Re}\int_0^{+\infty}(iu_t\bar{u}_x+u_{xx}\bar{u}_x)dx- \beta\,\text{Re}\int_0^{+\infty}|u|^2u\bar{u}_x dx.
\end{equation}
Now we proceed with the calculations of the integral terms in  the right hand of \eqref{eq1-a}. By integration by parts we have
\begin{equation}\label{03091}
\begin{split}
\text{Re}\int_0^{+\infty}\!\!\!iu_t\bar{u}_xdx&=-\text{Im}\int_0^{+\infty}\!\!\!u_t\bar{u}_xdx\\
&=-\text{Im}\int_0^{+\infty}\frac{d}{dt}(u\bar{u}_x)dx+\text{Im}\int_0^{+\infty}\!\!\!u\bar{u}_{xt}dx\\
&=-\text{Im}\int_0^{+\infty}\frac{d}{dt}(u\bar{u}_x)dx-\text{Im}\int_0^{+\infty}\!\!\!u_x\bar{u}_tdx-\text{Im}(u(0,t)\bar{u}_t(0,t)).
\end{split}
\end{equation}
Furthermore, from the equalities in the right-hand of \eqref{03091}  we conclude
$$-\text{Im}\int_0^{+\infty}u_t\bar{u}_xdx= -\text{Im}\int_0^{+\infty}\frac{d}{dt}(u\bar{u}_x)dx + \text{Im}\int_0^{+\infty}\!\!\!\bar{u}_xu_tdx  -\text{Im}(u(0,t)\bar{u}_t(0,t)),$$
and so
\begin{equation}\label{03092}
2\text{Im}\int_0^{+\infty}u_t\bar{u}_xdx= \text{Im}\int_0^{+\infty}\frac{d}{dt}(u\bar{u}_x)dx + \text{Im}(u(0,t)\bar{u}_t(0,t)).
\end{equation}
Then, combining the first equality in \eqref{03091} and \eqref{03092} we obtain
\begin{equation}\label{03093}
\text{Re}\int_0^{+\infty}iu_t\bar{u}_x dx=-\frac12\frac{d}{dt}\text{Im}\int_0^{+\infty}u\bar{u}_x dx-
\frac12\text{Im}(u(0,t)\bar{u}_t(0,t)).
\end{equation}
On the other hand, 
\begin{equation}\label{03094}
\begin{split}
\text{Re}\int_0^{+\infty}u_{xx}\bar{u}_xdx&=\frac12\int_0^{+\infty}(u_{xx}\bar{u}_x+\bar{u}_{xx}u_x)dx\\
&=\frac12 \int_0^{+\infty}\frac{d}{dx}|u_x|^2=-\frac12 |u_x(0,t)|^2.
\end{split}
\end{equation}
and lastly
\begin{equation}\label{03095}
\begin{split}
-\beta\,\text{Re}\int_0^{+\infty} |u|^2u\bar{u}_x dx&=-\frac{\beta}{2}\int_0^{+\infty}|u|^2(u\bar{u}_x+\bar{u}u_x)dx\\
&=-\frac{\beta}{4}\int_0^{+\infty}\frac{d}{dx}|u|^4dx=\frac{\beta}{4}|u(0,t)|^4.
\end{split}
\end{equation}

By combining \eqref{eq1}, \eqref{eq1-a}, \eqref{03093}, \eqref{03094} and \eqref{03095} we get
\begin{equation}\label{eq-moment-final}
2\alpha \int_{0}^{+\infty}(|u|^2)_xvdx=-\frac{d}{dt}\int_0^{+\infty}2\text{Im}(u\bar{u}_x)dx -\mathcal{Q}^{u}(t),
\end{equation}
with $\mathcal{Q}^{u}(t)$ defined in \eqref{i2-a}.

Finally, combining  \eqref{eq-moment-initial} with \eqref{eq-moment-final} and integrating on the time interval $[0, t]$ we  obtain the equation \eqref{i2} which governs 
the moment of solutions. 

\subsection{Formal derivation for the energy functional on the positive half-line}

By multiplying the first equation of \eqref{SK} by $\bar{u}_t$, taking the real part and integrating in $(0,+\infty)$ we get
\begin{equation}\label{tec11}
\text{Re}\int_0^{+\infty}u_{xx}\bar{u}_t\ dx= \alpha\, \text{Re}\int_0^{+\infty}u\bar{u}_tv\,dx+ \beta\, \text{Re}\int_0^{+\infty}|u|^2u\bar{u}_t\,dx.
\end{equation}
and now  we will calculate each integral above. 

An integration by parts gives
\begin{equation}\label{tec22}
\begin{split}
\text{Re}\int_0^{+\infty}u_{xx}\bar{u}_tdx=-\frac12 \frac{d}{dt}\int_0^{+\infty}|u_x|^2dx-\text{Re}(u_x(0,t)\bar{u}_t(0,t))
\end{split}.
\end{equation}

On the other hand, 
\begin{equation}\label{tec33}
 \beta\, \text{Re}\int_0^{+\infty}|u|^2u\bar{u}_t\,dx=\frac{\beta}{4}\frac{d}{dt}\int_0^{+\infty}|u|^4dx
\end{equation}

Finally, we have that 
\begin{equation}\label{tec44}
\begin{split}
\alpha\, \text{Re}\int_0^{+\infty}u\bar{u}_tvdx&=\frac{\alpha}{2}\int_0^{+\infty}v \frac{d}{dt}|u|^2 dx\\
&=\frac{\alpha}{2}\frac{d}{dt}\int_0^{+\infty}|u|^2vdx -\frac{\alpha}{2}\int_0^{+\infty}|u|^2v_tdx\\
\end{split}
\end{equation}
From the fundamental theorem of calculus and using  the second equation of \eqref{SK} we get 
\begin{equation}\label{tec55}
\begin{split}
-\frac{\alpha}{2}\int_0^{+\infty}|u|^2v_tdx&=-\frac{\alpha}{2}\int_0^{+\infty}\Big(\int_0^x(|u|^2)_x(x',t)dx'+|u(0,t)|^2\Big)v_tdx\\
&=-\frac{\alpha}{2}|u(0,t)|^2\int_{0}^{+\infty}v_tdx-\frac{\alpha}{2}\int_0^{+\infty}\Big[\int_0^x(|u|^2)_x(x',t)dx'\Big]v_tdx\\
&:=I(u,v) + II(u,v).
\end{split}
\end{equation}
where
\begin{equation}\label{tec55-a}
I(u,v)= -\frac{\alpha}{2}|u(0,t)|^2\Big(v_{xx}(0,t)+\frac12 v^2(0,t)-\gamma|u(0,t)|^2\Big)
\end{equation}
and 
\begin{equation}\label{tec55-b}
II(u,v)= -\frac{\alpha}{2\gamma}\int_0^{+\infty}\Big[\int_0^x\Big(v_t+v_{xxx}+\tfrac12(v^2)_x\Big)(x',t)dx'\Big]v_tdx.
\end{equation}
Now we  decomposed $II(u, v)$  as  follows: 
\begin{equation*}
\begin{split}
II(u, v)&=-\frac{\alpha}{2\gamma}\int_0^{+\infty}\Big[\int_0^x\Big[v_t(x',t)dx'\Big]v_tdx -\frac{\alpha}{2\gamma}\int_0^{+\infty}\Big[\int_0^x\Big(v_{xx}+\tfrac12(v^2)\Big)_xdx'\Big]v_tdx\\
&:=II_A(u,v) + II_B(u,v).
\end{split}
\end{equation*}
For the first term we have 
\begin{equation*}
\begin{split}
II_{A}(u, v)&=-\frac{\alpha}{4\gamma}\int_0^{+\infty}\frac{d}{dx}\Big(\int_0^x v_t(x',t)dx'\Big)^2dx\\
&=-\frac{\alpha}{4\gamma}\Big(\int_0^{+\infty}v_t(x, t)dx\Big)^2\\
&=-\frac{\alpha}{4\gamma}\Big[v_{xx}(0,t)+\frac12 v^2(0,t)-\gamma|u(0,t)|^2\Big]^2=-\frac12 \mathcal{E}_2(t),
\end{split}
\end{equation*}
with $\mathcal{E}_2(t)$ defined in \eqref{i3-b}. The second term is calculated as follows:

\begin{equation*}
\begin{split}
II_B(u,v)&=-\frac{\alpha}{2\gamma}\int_0^{+\infty}\!\!\Big(v_{xx}v_t+\frac12 v^2v_t\Big)dx + \Big(\frac{\alpha}{2\gamma}v_{xx}(0,t)+\frac{\alpha}{4\gamma}v^2(0,t)\Big)\int_0^{+\infty}\!\!\!v_tdx\\
&:=II_{B_1}(u,v) + II_{B_2}(u,v),
\end{split}
\end{equation*}
where
\begin{equation*}
\begin{split}
II_{B_1}(u,v)&=\frac{\alpha}{2\gamma}\int_0^{+\infty}v_xv_{xt}dx+\frac{\alpha}{2\gamma}v_x(0,t)v_t(0,t) - \frac{\alpha}{12\gamma}\frac{d}{dt}\int_0^{+\infty}\!\!v^3dx\\
&=\frac{d}{dt}\int_0^{+\infty}\Big(\frac{\alpha}{4\gamma}v_x^2 - \frac{\alpha}{12\gamma}v^3\Big)dx + \frac{\alpha}{2\gamma}v_x(0,t)v_t(0,t)
\end{split}
\end{equation*}
and 
\begin{equation*}
\begin{split}
II_{B_2}(u,v)&=\Big(\frac{\alpha}{2\gamma}v_{xx}(0,t)+\frac{\alpha}{4\gamma}v^2(0,t)\Big)\Big(v_{xx}(0,t)+\frac{1}{2}v^2(0,t)-\gamma|u(0,t)|^2\Big)\\
&=\frac{\alpha}{2\gamma}\Big(v_{xx}(0,t)+\tfrac{1}{2}v^2(0,t)\Big)\Big(v_{xx}(0,t)+\frac{1}{2}v^2(0,t)-\gamma|u(0,t)|^2\Big)\\
&=\mathcal{E}_2(t) + \frac{\alpha}{2}|u(0,t)|^2\Big(v_{xx}(0,t)+\frac{1}{2}v^2(0,t)-\gamma|u(0,t)|^2\Big),
\end{split}
\end{equation*}
with $\mathcal{E}_2(t)$ defined in \eqref{i3-b}. 

Then, from \eqref{tec55}, \eqref{tec55-a} and \eqref{tec55-b}, combined with the expressions obtained for $II_{A}(u,v)$, $II_{B}(u,v)$, $II_{B_1}(u,v)$ and $II_{B_2}(u,v)$ we have
\begin{equation}\label{tec66}
-\frac{\alpha}{2}\int_0^{+\infty}|u|^2v_tdx=\frac12 \mathcal{E}_2(t) + \frac{d}{dt}\int_0^{+\infty}\Big(\frac{\alpha}{4\gamma}v_x^2 - \frac{\alpha}{12\gamma}v^3\Big)dx + \frac{\alpha}{2\gamma}v_x(0,t)v_t(0,t) 
\end{equation}

Collecting the informations in \eqref{tec22}, \eqref{tec33}, \eqref{tec44} and \eqref{tec66} we finally obtain the law of the energy given in \eqref{i3}.

\end{document}